\documentclass[]{amsart}

\usepackage{color}
\usepackage{amsmath,amssymb,graphicx,MnSymbol}
\usepackage[utf8]{inputenc}
\usepackage[color,arrow,curve,matrix]{xy}
\usepackage{tikz-cd}
\def\xto#1{\xrightarrow[]{#1}}

\def\mspec{{\sf Spec}}
\def\k{{\mathsf K}}
\def\kSpec{\k{\sf Spec}}
\def\R{\mathbb{R}}
\def\C{\mathbb{C}}
\def\Adm{{\sf Adm}}
\def\fC{\mathfrak{C}}

\def\F{\mathfrak{F}}
\def\N{\mathbb{N}}
\def\Z{\mathbb{Z}}
\def\Grpd{\mathfrak{Grpd}}

\def\px{\mathfrak{x}}
\def\py{\mathfrak{y}}
\def\fR{\mathfrak{R}}

\def\U{{\mathcal U}}
\def\A{\mathbb{A}}
\def\RSpec{\mathbb{R}{\sf Spec}}
\def\id{{\sf Id}}
\def\Hom{{\sf Hom}}
\def\Spec{{\sf Spec}}

\def\Supp{{\sf Supp}}

\DeclareMathOperator\colim{\mathsf{colim}}
\DeclareMathOperator\tcolim{2\text{-}\colim}

\def\ide{{\sf Idem}}
\def\c{{\sf c}}
\def\Piv{{\sf Piv}}

\def\sP{{\sf P}}
\def\sl{\mathop{\alpha_{\sf sl}}}
\def\cG{\mathcal{G}}
\def\cD{\mathcal{D}}
\def\cC{\mathcal{C}}
\def\cF{\mathcal{F}}

\newcommand{\cal}[1]{\mathcal{#1}}

\def\r{\mathfrak{r}}
\def\t{\mathfrak{t}}
\def\p{\mathfrak{p}}

\def\m{\mathfrak{m}}
\newcommand{\gen}[1]{\left\langle #1 \right\rangle}
\newcommand{\x}{\mathsf{x}^{\mathsf{o}}}
\def \pt{\mathop{\bullet}\nolimits}

\newtheorem{Pro}{Proposition}[subsection]
\newtheorem{Le}[Pro]{Lemma}
\newtheorem{Th}[Pro]{Theorem}
\newtheorem{Co}[Pro]{Corollary}
\theoremstyle{definition}

\theoremstyle{remark}
\newtheorem{Exm}[Pro]{Example}

\newtheorem{Rem}[Pro]{Remark}

\def\fb{{\mathfrak B}}

\title{The fundamental group of binoid varieties}
\author{Holger Brenner$^1$ \& Ilia Pirashvili$^2$}
\address{$^1$ - Institut f{\"u}r Mathematik, Universit{\"a}t Osnabr{\"u}ck, Albrechtstra{\ss}e 28a, 49076 Osnabr{\"u}ck, Germany.}
\email{holger.brenner@uni-osnabrueck.de}
\address{$^2$ - Institut f{\"u}r Mathematik, Universit{\"a}t Osnabr{\"u}ck, Albrechtstra{\ss}e 28a, 49076 Osnabr{\"u}ck, Germany.}
\email{ilia.pirashvili@uni-osnabrueck.de}
\begin{document}

\maketitle

\begin{abstract} Binoid schemes generalise monoid schemes, which in turn enable us to generalise toric varieties. Let $X$ be a binoid scheme. The aim of this paper is to calculate the topological fundamental group of $\k X$, where $\k=\C$ or $\R$. For the latter, we will give an explicit way of calculating the fundamental group using methods from 2-category theory. Indeed, we will calculate the more general fundamental groupoid.

As a specialisation, we will also look at the Stanley Reisner Rings. Our method simplifies in this case, allowing us to describe the fundamental groupoid in terms of the simplicial complex directly.
\end{abstract}

\section{Introduction} The notion of a monoid scheme was introduced by Deitmar \cite{deitmar} in 2005, following the works of Kato \cite{kato}, who proved the link between monoid schemes and toric varieties in 1994. This link was further explored in \cite{chww}. Indeed, monoid schemes are closely connected with many areas of mathematics, such as algebraic k-theory \cite{chww}, tropical geometry \cite{jep}, logarithmic geometry \cite{lgm}, \cite{lsf}, absolute mathematics \cite{kow} and most notably $\mathbb{F}_1$-geometry \cite{cc1}.

A binoid is a monoid with an absorbing element, which we denote by $0$, as we will mostly stick with the multiplicative notation in this paper. Further, all binoids and monoids are commutative and finitely generated throughout this paper. We are able to construct binoid schemes, much in the same way we construct monoid schemes. For a commutative ring $\k$ and a binoid $M$, we can consider the $\k$-algebra $\k[M]$. Likewise, for a binoid scheme $X$, we can consider $\k[X]$. Further details to these constructions can be found in \cite{chww}, but the reader is advised that, while a monoid in their terminology has an absorbing element, it need not be preserved by a homomorphism.

There are many reasons to be interested in binoids rather than monoids. Chief amongst them is that they are more general and include many important classes, such as the Stanley-Reisner Rings, while still being within the scope of the same tools. Another reason in our case is because it also helps in our proof, see Section \ref{binoidsuse}.

A very closely related, but distinct, construction is to consider $\kSpec(M):=\Hom(M,\k)$. Here, $\k$ is seen as a multiplicative binoid. This is a subset of $\k^n$, with $n$ being the number of generators of $M$. We have $\kSpec(M)\simeq \kSpec(\k[M])$, where the latter denotes the set of $\k$-algebra homomorphisms from $\k[M]$ to $\k$. For example, if $M:=\gen{x}\simeq \N$, $\kSpec(M)\simeq \k$, whereas $\k[M]\simeq \k[x]$. It is basically the zeroset of our binoid. Naturally, we can also consider $\k X$, for a binoid scheme $X$. It is the set of $\k$-valued points of the scheme $\k[X]$.

If $\k$ is equipped with a topology, $\k X$ will inherit it. In this paper, we will focus on the cases when $\k=\R,\C$, the real or complex numbers. Our aim is to calculate $\pi_1(\k X)$, where the topology of $\k X$ is induced by the topology of $\k^n$. We will actually calculate the more general $\Pi_1(\k X)$, where $\Pi_1$ denotes the fundamental groupoid. The latter is defined to be a category, whose objects are the points of the given topological space, and whose morphisms are homotopy classes of paths. It can be shown readily that this is equivalent to taking the fundamental group at every connected component simultaneously. As such, it carries the information of $\pi_0$ and $\pi_1$ at every connected component.

One of the main advantages of the fundamental groupoid is that the Seifert-van Kampen theorem no longer requires a common point. This allows us to use the affine covering of a binoid scheme $X$, to calculate the fundamental group(oid) of $\k X$.

As such, our focus will become the affine case $\kSpec(M)$. One of our main theorems is the reduction of the computation of the homotopy of general binoids to grouplike binoids in Theorem \ref{reductiontogroupsnonsep}. By a grouplike binoid, we mean a group with an absorbing element added.

This is proven in several steps. We first prove it in the special case when $M$ is integral and separated. Then, we will remove these conditions one after the other using induction arguments, primarily on the number of prime ideals.

However, we can simplify things further when $\k=\R$. Theorems \ref{Seifert-vanKampen} and \ref{colim=2-colim} allows us to reduce the calculation of the fundamental group(oid) of a topological space $X$ to simple combinatorics, if $X$ admits a discrete covering which we understand in full. By discrete, we mean a covering $\{U_i\}$ of $X$, where every $\Pi_1(U_i)$, every $\Pi_1(U_{ij})$ and every $\Pi_1(U_{ijk})$ is a discrete groupoid. That is, a groupoid where the automorphism group of every point is trivial.

Equivalently, $\{U_i\}$ is a discrete covering if the fundamental group of every $U_i$, $U_{ij}$ and $U_{ijk}$ is trivial at every point. Note that, however, $\pi_0$ need not be trivial.

As it turns out, $\Pi_1(\RSpec(M))$ is a discrete groupoid by Corollary \ref{pi0group}, for every binoid $M$. We are also able to understand the functors between the fundamental groupoids, induced by the gluing homomorphisms, see Subsection \ref{morps}. Hence, every quasi-separable binoid scheme has a natural discrete covering. This enables us to give a method for calculating the fundamental group(oid) of real binoid schemes.

This calculation is especially clear in the case of the Stanley-Reisner Rings. These are rings associated to a simplicial complex $\Delta$. Of course, the fundamental group of $\R(\Delta)$ can also be calculated in other ways. The aim of our calculations is to give a demonstration of our proposed methods in a relatively simpler case.

Ancillary, this also enables us to simplify the calculation of the homology of $\R X$ in Subsection \ref{homology}.

Our results are closely connected to the works by V. Uma \cite{Vuma}, where the author focuses on the fundamental group of real toric varieties, coming from smooth fans. However, even in that case, our methods are completely different, and the general idea of calculating the fundamental groupoid from a discrete covering could have independent uses.

\section{The $k$-Spectrum of finitely generated Binoids}\label{affinecase}

\subsection{Preliminaries}\label{msl}

A \emph{binoid} is a set $M$, together with a binary, associative operation $\times:M\times M\rightarrow M$ and two distinguished elements: An identity element $1\in M$, with $1m=m$ for all $m\in M$, and an absorbing element $0\in M$, with $0m=0$ for all $m\in M$. In the additive notation, the identity is denoted by $0$ and the absorbing element by $\infty$. We will, however, stick with the multiplicative notation. As one would assume, this absorbing element is preserved by binoid homomorphisms.  The category of commutative binoids is denoted by ${\bf bn}$. 
Throughout this paper, all binoids and monoids are assumed to be finitely generated, commutative and written using the multiplicative notation, unless specified otherwise.

A binoid is said to be \emph{grouplike} if every nonzero element is invertible. The set of all invertible elements of a binoid $M$, together with the absorbing element $0$, is denoted by $M^{\x}$. The associated grouplike binoid to any group $G$ will be denoted by $G^{\mathsf{o}}$ in this paper.

The set of idempotents of a binoid (or monoid) $M$ is denoted by $\ide(M)$. A binoid is called \emph{separated} if $x=xy$, $x,y\in M$ implies that $x=0$ or $y=1$. Recall that a binoid is called \emph{integral} if for all $xy=0$ in $M$, either $x=0$ or $y=0$. The nonzero elements $M^\bullet$ of such a binoid $M$ form a monoid.

We will need some special quotients of binoids and monoids in this paper. For an ideal $I$ of a binoid (or monoid) $M$ (that is to say, $0\in I$ and $IM\subseteq M$), we denote by $M/I$ the \emph{Rees quotient} (also called the \emph{Rees factorisation}), which is $(M\setminus I)\cup \{0\}$ as a set. The multiplication in $M/I$ is defined in such a way that $q:M\rightarrow M/I$ is a homomorphism. Indeed, it is universal with respect to that property. Here, $q(x)=x$ if $x\not \in I$ and $q(x)=0$, if $x\in I$.

We mention two special cases of this factorisation: i) we can take $I$ to be the set of all nilpotent elements and ii) we can take $I$ to be the principal ideal generated by $a$, so $I=aM=(a)$.  In the first case, the corresponding quotient is denoted by $M_{\sf red}$, while in the second case we use the notation $M/(a\sim 0)=M/(a)$.

For a submonoid $S\subseteq M$ of $M$, we can consider the quotient $M/(S\sim 1)$. Two elements $x,y\in M$ are equated under this congruence if and only if there are $s,t\in S$, such that $xs=yt$. This is the smallest congruence where $S\sim 1$, hence our notation.

In particular, for an element $a$, we can consider the submonoid $\gen{a}$ generated by $a$. In other words, we consider the quotient of $M$ by the congruence which equates $x$ and $y$ in $M$ if and only if $xa^m=ya^n$ for some $m,n\geq 0$. We denote this by $M/(a\sim 1)$ for obvious reasons.

If $M$ is a monoid, we denote by $M^{\sf sl}$ the quotient of $M$ by the congruence $\sim_{\sf sl}$, for which $a\sim_{\sf sl} b$ if and only if $a^m=bc$ and $b^n=ad$ for some $c,d\in M$ and $m,n\geq 1$. This is the associated semi-lattice of a monoid \cite{Grillet}, and is often called its \emph{booleanisation}, see for example \cite{S.Bot}. In $M^{\sf sl}$ all elements are idempotents and this is a minimal quotient with this property. That is, $M^{\sf sl}=M/(m\sim m^2)$. The natural map $M\rightarrow M^{\sf sl}$ is denoted by $\sl$.

We summarise some properties of these quotients which we will need later.

\begin{Le}\label{quot1} Let $M$ be a binoid.
\begin{itemize}
\item [i)] The natural homomorphism $q_1:M\rightarrow M/(a)$ yields a surjective map $\ide(M)\rightarrow \ide(M/(a))$ for any element $a\in M$.  Moreover, the preimage of any nonzero idempotent of $\ide(M/(a))$ in  $\ide(M)$ is  a singleton.

\item [ii)] Let $e\in \ide(M)$ be an idempotent element. The natural homomorphism $q_2:M\rightarrow M/(e\sim 1)$ yields a surjective map $\ide(M)\rightarrow \ide(M/(e\sim 1))$. 

\item [iii)] Assume $a\in M$ is invertible. Then the natural homomorphism  $q_2:M\rightarrow M/(a\sim 1)$ yields a surjective map $M^{\x}\rightarrow (M/(a\sim 1))^{\x}$.

\item [iv)] One has $(M_{\sf red})_{\sf red}=M_{\sf red}.$ Moreover, the canonical map $M\rightarrow M_{\sf red}$ is bijective on the unit group.
\end{itemize}
\end{Le}

\begin{proof} i) Let $z\in M/(a)$. Since $q_1$ is a surjection, we may assume that $z=q_1(x)\in \ide(M/(a))$. If $x\in aM$, $z$ is the image of the idempotent $0$.  If $z\neq 0$, this $x\in \ide(M)$ is the only idempotent of $M$, with $z=q_1(x)$, as $q_1$ restricted to $M\setminus(a)$ is bijective.

ii) Take an element $y\in M$, such that $q_2(y)$ is an idempotent in $M/(e \sim 1)$. Then $ye^m=y^2e^n$ for some $m,n\geq 0$. We have to find an idempotent $u\in M$, with $q_2(u)=q_2(y)$. Since $e$ is an idempotent, we can assume that $0\leq m,n\leq 1$. There are four cases to consider. If $(m,n)=(0,0)$, take $u=y$. In all other cases, take $u=ye$.

iii) Let $x\in M$ with $q_2(x)\in (M/(a\sim 1))^{\x}$. We may assume this since $q_2$ is surjective. There exists $y\in M$ and $m,n\geq 0$ such that $xya^m=a^n$ since $q_2(x)$ is invertible. It follows that $x\in M^{\x}$ since $a$ is invertible.

iv) Denote by $q$ the canonical surjection $M\rightarrow M_{\sf red}$. Assume $q(x)$ is nilpotent in $M_{\sf red}$. There exists $n\in \N$, such that $q(x)^n=0$. That is, $q(x^n)=q(0)$ and hence, $x^n$ is nilpotent in $M$. As such, $x$ is also nilpotent and $q(x)=0$.  The last assertion is obvious.
\end{proof}

\begin{Le}\label{quot2} Let $M$ be a monoid and $a\in M$.
\begin{itemize}

\item [i)] If $f:M\rightarrow M_1$ is a surjective homomorphism of monoids, the induced map $f^{\sf sl}:M^{\sf sl}\rightarrow M_1^{\sf sl}$ is surjective.

\item [ii)] The homomorphism $M^{\sf sl}\rightarrow (M/(a))^{\sf sl}$, induced by the natural map $q_1:M\rightarrow M/(a)$, is an isomorphism if and only if $a$ is nilpotent.

\item [iii)] The homomorphism $M^{\sf sl}\rightarrow (M/(a\sim 1))^{\sf sl}$, induced by the natural map $q_2:M\rightarrow M/(a\sim 1)$, is an isomorphism if and only if $a$ is invertible.
\end{itemize}
\end{Le}

\begin{proof} i) Obvious since $M^{\sf sl}$ is a quotient of $M$. 

ii) Say $a$ is nilpotent. That is, $a^n=0$ for some $n$. It follows that $a\sim_{\sf sl} 0$, implying that $q_1$ is injective. Bijectivity then follows from i).

On the other hand, if $q_1^{sl}$ is injective, it follows that $a\sim_{\sf sl} 0$ in $(M/(a))^{\sf sl}$ since $\sl(a)$ and $\sl(0)=0$ have the same images in $(M/(a))^{ sl}$. Thus $\sl(a)=0$, showing that $a$ is nilpotent.

iii) The quotient $m\sim m^2$ equates the subgroup of invertible elements to $1$. (This can be readily checked by the construction of the quotient $\sim_{\sf sl}$). As such, $M^{\sf sl}\rightarrow (M/(a\sim 1))^{\sf sl}$ is injective if $a$ is invertible.

Conversely, assume $M^{\sf sl}\rightarrow (M/(a\sim 1))^{\sf sl}$ is injective. The elements $a$ and $1$ have the same image in $(M/(a\sim 1))^{\sf sl}$ since they already are the same in $M/(a\sim 1)$. It follows that $\sl(a)$ and $\sl(1)$ have the same image in $M^{\sf sl}$. By the definition of the congruence $\sim_{\sf sl}$, $a$ is invertible.
\end{proof}

\subsection{On prime ideals}\label{1277prime}

Let us recall that an ideal $\p\neq M$ of a binoid $M$ is called \emph{prime} if $xy\in \p$, $x,y\in M$ implies that either $x\in \p$ or $y\in\p$. The set of all prime ideals is denoted by $\mspec(M)$. It is a poset with respect to inclusion and is empty if and only if $0=1$ in the binoid. In all other cases, it has a unique maximal ideal $\max(M)$, consisting of all non-invertible elements. If $M$ is integral, the set $\{0\}$ is prime and is clearly the smallest element of $\mspec(M)$. It follows that $\mspec(M)$ has only one element if and only if $M$ is grouplike.

There exists a natural bijection \cite{Holger}, \cite{spec}
$$\mspec(M)\xto{\simeq} \Hom_{\bf bn}(M,\{0,1\})$$
which sends a prime ideal $\p$ to the binoid homomorphism
$\chi_\p:M\to\{0,1\}$, defined by $$\chi_\p(m)=\begin{cases} 0, \ {\rm if} \ m\in \p,\\1, \ {\rm if} \ m\not \in \p.\end{cases}$$ 
If $M$ is finitely generated, then $\mspec(M)$ is finite. Moreover, $\mspec(M)\simeq(M^\bullet)^{\sf sl}$ if additionally $M$ is integral (see \cite{Holger}, \cite{spec}). Hence, Lemma \ref{quot2} can be used to estimate the cardinality of $\mspec(M)$. For example, if $a$ is not invertible in $M$, then
$$|\mspec(M)|>|\mspec(M/(a\sim 1))|$$
and if $a$ is not nilpotent, then
$$|\mspec(M)|>|\mspec(M/(a))|.$$

\subsection{Pivotal elements}\label{dis77}

For an element $x\in M$ we set
$$M^x:=\{y\in M| xy=x\}.$$
It is clear that $M^x$ is a submonoid of $M$. Call $x$ \emph{pivotal} if $x\neq 0$ and $M^x\neq\{1\}$. So, $M$ is separated if and only if $M$ has no pivotal elements. The set of all such elements is denoted by $\Piv(M)$. We also set
$$\sP(M):=\bigcup_{x\in \Piv(M)} M^x$$

\begin{Le}\label{d131.77} Let $M$ be an integral binoid. 
\begin{itemize}
	\item[i)] If $x\in \Piv(M)$ and $z\not =0$, then $xz\in \Piv(M)$. So, $\Piv(M)$ is an ideal of $M^\bullet$.
	\item[ii)] The set $\sP(M)$ is a subsemigroup of $M^\bullet$. Moreover, $\sP(M)$ is a submonoid if and only if $M$ is not separated.
	\item[iii)] Let $a$ be an invertible element of $M$ and $M'=M/(a\sim 1)$. For any element $x'\in \Piv(M')$, there exist $x\in \Piv(M)$, such that $q_2(x)=x'$ and for any $y'\in \sP(M')$, there exists $y\in \sP(M)$ such that $q_2(y)=y'$.
\end{itemize}
\end{Le}

\begin{proof} i) If $1\not =y\in M^x$, then $y\in M^{xz}$. Hence, $xz\in \Piv(M)$.

ii) If $y_1\in M^{x_1}$ and $y_2\in M^{x_2}$, then $y_1y_2\in M^{x_1x_2}.$  For the last assertion, observe that $1\in \sP(M)$ if and only if $\Piv(M)\not =\emptyset$. 

iii) Assume $x'y'=x'$ holds in $M'$ and $x'\not =0$, $y'\not =1$. Choose $x,z\in M$ such that $q_2(x)=x',q(z)=y'$. Then, $xza^m=xa^n$ for some $m,n\geq 0$. We obtain $xy=x$ since $a$ is invertible, where $y=za^{m-n}$ and $q(y)=q(z)=y'$.
\end{proof}

\subsection{The space $\kSpec(M)$}

For a binoid $M$ and a commutative ring $\k$, we denote by $\kSpec(M):=\Hom_{\bf bn}(M,\k)$ and call it the $\k$-\emph{spectrum} of $M$.

Assume $M$ is finitely generated with $\{e_1,\ldots, e_n\}$ as generators. Every morphism of binoids is defined by the values on its generators, which gives rise to the following embedding of $\kSpec(M)$ in $\k^n$:
$$\kSpec(M)  \ni  f\mapsto (f(e_1),\ldots, f(e_n))\in \k^n.$$
Thus, it can be thought of as the `set of zeros' of the associated binoid ring $\k[M]$. For $\k=\R, \mathbb{C}$, the binoid $\kSpec(M)$ inherits a subspace topology, which is independent of our choice of generators. In what follows, $\kSpec(M)$ is considered with the subspace topology of $\R^n$ and $\mathbb{C}^n$.

\begin{Le} \label{131.77} Let $M$ be a binoid and $\k$ a commutative ring.
\begin{itemize}
	\item[i)] One has
	$$\kSpec(M)=\k\textnormal{-}\Spec(\k[M]),$$
	where $\k\textnormal{-}\Spec(\k[M])$ is the set of the classical $\k$-valued points $\Hom_{\k\textnormal{-}\mathsf{Alg}}(\k[M],\k)$.
	\item[ii)] Let $\k$ be an integral domain $\k$. Than
	$$\kSpec(M)=\kSpec(M_{\sf red})$$
	\item[iii)] Let $M$ be integral. Then $\kSpec(M)$ is a topological monoid.
\end{itemize}
\end{Le} 

\begin{proof} i) This is clear, since $M\rightarrow \k$ gives rise to a $\k$-algebra homomorphism $\k[M]\rightarrow \k$.

ii) Observe  that any binoid homomorphism takes nilpotent elements to nilpotent ones. Since $\k$ has no nonzero nilpotent element, any binoid homomorphism $M\rightarrow \k$ factors through $M_{\sf red}$ and the result follows.

iii) Since $M$ is integral, one has
$$\kSpec(M)=\Hom_{\bf Mon}(M^\bullet, \k).$$ 
The result follows from the fact that the set of all (commutative) monoid homomorphisms has a natural monoid structure.
\end{proof}

\begin{Rem}	Recall that a space $X$ is $n$-\emph{simple}, if $\pi_1X$ is abelian and acts trivially on $\pi_nX$, $n\geq 1$. The above lemma shows that $\kSpec(M)$ is in particular an $H$-space for an integral binoid $M$. Hence, it is $n$-simple for all $n\in \mathbb{N}$.

This is, however, no longer true in the non-integral case, as the set of binoid homomorphisms need not have a unit.
\end{Rem}

\subsection{Decomposition of a binoid by its idempotents}

Let $M$ be a commutative binoid. To simplify notations, we write $\Adm(M)$ instead of $\mspec(\ide(M))$. For an element $\r\in \Adm(M)$, we set $\r^{\sf c}=\ide(M)\setminus \r$. Thus for such $\r$ one obtains  a decomposition
$$\ide(M)=\r\amalg\r^{\sf c}.$$
Clearly $0\in \r$ and $1\in \r^{\sf c}$. Let $e\in\r$ and $m\in M$. Assume $e':=me$ is an idempotent. From $ee'=eme=me=e'$ and the fact that $\r$ is an ideal, we see that $e'\in \r$ as well. Hence, the only idempotents of the ideal $\r M\subseteq M$ are elements of $\r$. Though $\r$ is a prime ideal of $\ide(M)$, $\r M$ need not be a prime ideal.

We let $M(\r)$ be the quotient of $M$ by the relations $\r M\sim 0$ and $\r^{\sf c}\sim 1$. Thus, we can write
$$M({\r})=M/(\r M, \ \r^{\sf c}\sim 1).$$
The same monoid can be described as a push-out diagram in the category of binoids:
$$\xymatrix{\ide(M)\ar[d]^{\chi_\r}\ar[r] & M\ar@{-->}[d] \\ \{0,1\}\ar@{-->}[r] & M(\r)}$$

\begin{Le}\label{componentreduction} Let $M$ be finitely generated. One has
$$\kSpec(M)=\coprod\limits_{\r\in\Adm(M)} \kSpec(M(\r))$$
for any binoid $M$ and any integral domain $\k$, where the binoids $M(\r)$ have only trivial idempotents.
\end{Le}

\begin{proof} Take  $\phi:M(\r) \to \k$, where $\r\in \Adm(M)$ and precompose it with the canonical homomorphism $M\rightarrow M(\r)$ to obtain an element in $\kSpec(M)$. Conversely, take an element $\psi:M\rightarrow \k$ from the LHS. It sends any idempotent to either $0$ or $1$. Denote by $\r$ the set of idempotents $e\in {\sf Idem}(M)$, for which $\psi(e)=0$. Then $\r\in \Adm(M)$ and  $\psi$ factors through $M(r)$. The claim follows.
	
For the last assertion, observe that one can use  Lemma \ref{quot1} step by step to show that any idempotent of $M(\r)$ has a lifting to an idempotent of $M$. The claim follows as any idempotent of $M$ maps to either $0$ or $1$ in $M(\r)$.
\end{proof}

We now consider the functorial behaviour of this decomposition. Let $f:M\to M'$ be a binoid  homomorphism. Since $\ide$ (resp. $\mspec$) is a covariant (resp. contravariant) functor, we see that $\Adm(-):=\mspec(\ide(-))$ is a contravariant functor. Thus, for any $\r'\in \Adm(M')$, we have $$f^\smalltriangleleft(\r')=\{e\in \ide(M)| f(e)\in \r'\}\in \Adm(M),$$
and there exists a unique homomorphism of  binoids $f_{\r}:M(f^\smalltriangleleft(\r'))\to M'(\r'))$ for which the diagram 
$$\xymatrix{ M\ar[r]\ar[d]_f & M(f^\smalltriangleleft(\r'))\ar[d]^{f_{\r}}\\
			 M'\ar[r] & M'(\r') }$$
commutes. Here, the horizontal arrows are the canonical maps to the quotient. To see this, observe that $M'(\r')=M'/\mathcal{K}'$. where $\mathcal{K}'$ is the congruence given by $\r\sim 0$ and $\r^{\sf c}\sim 1$. On the other hand, $M(f^\smalltriangleleft(\r'))=M/\mathcal{K}$, where $\mathcal{K}$ is given by $\{e\in \ide(M)|f(e)\in\r\}\sim 0$ and $\{e\in\ide(M)|e\in\r^{\sf c}\}\sim 1$. The induced diagram
$$\xymatrix@!C{\kSpec(M') \ar[d]_{f^*}\ar[r]^{=\hspace{2em}} &  \coprod\limits_{\r'\in\Adm(M')} \kSpec(M'(\r'))\ar[d]^{f^*_{\r}}\\
\kSpec(M) \ar[r]^{=\hspace{2em}} &  \coprod\limits_{\r\in\Adm(M)} \kSpec(M(\r))}$$
also commutes, where the right vertical morphism maps the component corresponding to $\r'$ to the component corresponding to $f^\smalltriangleleft(\r')$.

\subsection{Comparison maps and the main theorem}\label{binoidsuse}

Recall that  $M^{\x}$ denotes the set of all invertible elements of a binoid $M$, together with the absorbing element $0$. We have the following morphisms in the category ${\bf bn}$ of binoids:
$$M^{\x}\xrightarrow{i} M\xrightarrow{j}M^{\x}.$$
Here, $i$ is the natural inclusion, while $j$ is the canonical projection onto the quotient $M^{\x}\simeq M/M_+$. We denote by $M_+$ the maximal ideal of $M$, which consists of the non-invertible elements of $M$. Concretely, the homomorphism $j$ is given by
$$j(m)=\begin{cases} m,  & {\rm if} \ m\in M^{\x}\\0,& {\rm if} \ m\not \in M^{\x}.\end{cases}$$ 
Since $j\circ i=\id_{M^{\x}}$, we have $i^*\circ j^*=\id_{\kSpec(M^{\x})}$ for the induced continuous maps
$$\kSpec(M^{\x})\xleftarrow{i^*} \kSpec(M)\xleftarrow{j^*} \kSpec(M^{\x}).$$
This enables us to identify the space $\kSpec(M^{\x})$ as a subspace of $\kSpec(M)$ via $j^*$. A point $p: M\rightarrow \k$ of the space $\kSpec(M)$ belongs to $\kSpec(M^{\x})$ under this identification if and only if $p(m)=0$ for any $m\not \in M^{\x}$.

We can now state the main result of this section.

\begin{Th}\label{reductiontogroupsnonsep} Let $\k=\R,\C$ and $M$ be a finitely generated binoid. There exists a homotopy equivalence
$$\kSpec(M)\simeq\coprod\limits_{\r\in \Adm(M)} \kSpec((M(\r)^{\x}),$$
which is compatible with the decomposition given in \ref{componentreduction}, and which is induced by $j^*_\r:\kSpec((M(\r))^{\x})\rightarrow \kSpec(M(\r))$ on each summand.
\end{Th}

The proof is based on several reductions. We will first check it under the assumption that $M$ is integral and separated in Section \ref{527,77}. We will then use an induction argument on the number of prime ideals to prove the Theorem in its full generality.

The latter part itself will be done in two independent steps. In the first one, we will get rid of the assumption that $M$ is separated, using another induction argument. Finally, we will also get rid of the assumption that $M$ is integral.

\subsection{The case of separated and integral binoids}\label{527,77} 

We start by recalling the following result:

\begin{Le}[\cite{S.Bot}, Theorem 5.2.7]\label{grading} Let $M$ be an integral and separated binoid. There exists a local monoid homomorphism
$$\delta: M^\bullet \rightarrow \N,$$
called the \emph{grading} of $M$. By local we mean that non-invertible elements map to non-invertible elements, and hence, $\delta(x)=0$ if and only if $x$ is invertible in $M$.
\end{Le}

\begin{proof} This was proven in \cite[Theorem 5.2.7]{S.Bot} for the case when $M$ has no non-trivial invertible elements. As such, all we have to do is compose it with the canonical surjective homomorphism $q:M\rightarrow M/M^{\x}\sim 1$. Here $q(x)=q(y)$ if and only if there exists an invertible element $u\in M$, such that $y=ux$. It is clear that $M/M^{\x}\sim 1$ has no non-trivial invertible elements, reducing it to the cited theorem.
\end{proof}

The next result shows that Theorem \ref{reductiontogroupsnonsep} is true for integral and separated binoids.
\begin{Pro}\label{reductiontogroup} Let $\k=\R, \mathbb{C}$ and $M$ be an integral and separated binoid. The topological space $\kSpec(M^{\x})$ is a deformation retract of $\kSpec(M)$.
\end{Pro}

\begin{proof} We have $\kSpec(M)=\Hom_{\bf Mon}(M^\bullet, \k)$ since $M$ is integral. Consider the monoid homomorphism
$$M^\bullet\rightarrow M^\bullet\times\N,$$
given by $m\mapsto(m,\delta(m))$, where $\delta$ is the grading given in Lemma \ref{grading}. A $\k$-Point of $M^\bullet\times\N$ is the same as a $\k$-point $p$ of $M^\bullet$ and a $\k$-point of $\N$. The latter corresponds to an element of $\k$. The $\k$-point $(p,t):M^\bullet\times\N\rightarrow \k$ sends $(m,e)$ to $t^ep(m)$. This induces the continues map
$$\kSpec(M)\times \mathbb{A}^1\rightarrow \kSpec(M),$$
which sends $(p,t)$ to the monoid homomorphism
$$F(p,t):M^\bullet \rightarrow \k,$$
given by
$$F(p,t)(m)=t^{\delta(m)}p(m).$$
Assume that $0^0=1$. We obtain the map
$$F:\kSpec(M)\times [0,1]\rightarrow \kSpec(M).$$
Clearly, $F(p,1)=p$ and hence, $F(-,1)=\id_{\kSpec(M)}$. On the other hand, if $t=0$ we have
$$F(p,0)(m)=\begin{cases} 0, & {\rm if} \ \delta(m)\not =0\\ p(m), & {\rm if} \ \delta(m)=0.\end{cases}$$
But, $\delta(m)=0$ if and only if $m$ is invertible. We rewrite the above as
$$F(p,0)(m)=\begin{cases} 0, & {\rm if} \ m \ {\rm is  \ not  \ invertible}\\ p(m), & {\rm if} \ m \ \ {\rm is    \ invertible}. \end{cases}$$
This shows that $F(p,0)\in \kSpec(M^{\x})$ and if $p\in \kSpec(M^{\x})$, we have $p(m)=0$ for all non-invertible $m$. Thus, $F(p,0)=p$ for such $p$. This finishes the proof.
\end{proof}

\begin{Rem}	 It follows from the proof that the map $j^*:\kSpec(M^{\x})\to \kSpec(M)$ is  a homotopy equivalence. Since $i^*\circ j^*=\id_{\kSpec(M^{\x})}$, we see that $i^*:\kSpec(M)\to \kSpec(M^{\x})$ is also a homotopy equivalence.
\end{Rem}

\subsection{Proof of Theorem \ref{reductiontogroupsnonsep}} We will use induction on $\ell=|\mspec(M)|$, the number of prime ideals of our binoid $M$. The base step $\ell=1$ holds because in this case, the maximal ideal consists of nilpotent elements only. Hence, the reduction $M_{\sf red}$ of $M$ is grouplike, see Section \ref{1277prime}.

Assume that Theorem \ref{reductiontogroupsnonsep} is true for all binoids $M$ for which $|\mspec(M)|<\ell$. We may assume that $0$ is the only nilpotent element of $M$, and that it only has trivial idempotents, using Lemmas \ref{131.77} and \ref{componentreduction}.

\subsubsection{Separatedness}\label{sep}

We want to get rid of the assumption that $M$ is separated, but we need to introduce some notations to proceed further.
We showed that Theorem \ref{reductiontogroupsnonsep} is true if $M$ is separated in Proposition \ref{reductiontogroup}. As such, we can assume that $\Piv(M)\neq\emptyset$. Take any $x\in \Piv(M)$ and $y\neq 1\in M^x$. Hence, $xy=x$ and $x\not =0, y\neq 1$. In particular, $x$ is not invertible. Since $M$ has only trivial idempotents, $y\neq x$. Set 
$$M_0:=M/(x), \quad M_1:=M/(y\sim 1), \quad M_{01}:=M_0/(y\sim 1)\simeq M_1/(x).$$

We have the following commutative diagram of binoids
$$\xymatrix{ & & M_1\ar[rd] &\\
			 M\ar[rru]\ar[rd] & & & M_{01},\\
			 & M_0\ar[rru] & & }$$
where the arrows are the natural projections onto the quotients. It follows from part i) of Lemma \ref{quot1} that $M_0$ has no non-trivial idempotents in this case.

We claim that the induced map $\ide(M_1)\to \ide(M_{01})$ is a bijection. It is surjective by part 1 of Lemma  \ref{quot1} and all the fibres are singletons, except maybe the preimage of zero. Take an idempotent $q_1(a)\in M_1$ which maps to zero in $M_{01}$. Then, $q_1(a)=q_1(x)q_1(b)$ in $M_1$, for some $b\in M$. It follows that there are $n,m\geq 0$ such that $x^2b^2y^m=xby^n$. We have $xy^k=x$ for all $k$ as $xy=x$ and so $x^2b^2=xb$, making $xb$ is an idempotent of $M$. Our assumption implies that either $xb=0$, or $xb=1$. The last equality is impossible, because $x$ is not invertible. Hence, $xb=0$ and $q_1(a)=0$, yielding the claim.

Since the assignments $M\mapsto \kSpec(M)$ and $M\mapsto \coprod_\r\kSpec(M(\r)^{\x})$ are (contravariantly) functorial, one obtains the following commutative diagram:

$$\xymatrix{ & & \coprod_{\r}\kSpec(M_{1}(\r)^{\x})\ar[dr]\ar@<-.5ex>[dd]_{j_1^*} &\\ \coprod_{\r}\kSpec(M_{01}(\r)^{\x})\ar[dr]\ar[urr]\ar@<-.5ex>[dd]_{j_{01}^*} & & & \kSpec(M^{\x})\ar@<-.5ex>[dd]_{j_1^*}\\
& \kSpec(M_0^{\x})\ar[urr]\ar@<-.5ex>[dd]_{j_1^*} & \kSpec(M_1)\ar@<-.5ex>[uu]_{i_1^*}\ar[dr] &\\
\kSpec(M_{01}) \ar@<-.5ex>[uu]_{i_{01}^*} \ar[urr]\ar[dr] & & & \kSpec(M) \ar@<-.5ex>[uu]_{i^*}\\
& \kSpec(M_0).\ar[urr]\ar@<-.5ex>[uu]_{i_0^*} & & }$$
Here, $\r$ is running through the set of admissible decompositions of $\ide(M_1)=\ide(M_{01})$. One has $|\mspec(M_0)|<\ell$ according to Lemma \ref{quot2} and Section \ref{1277prime}. It follows that $|\mspec(M_{01})|\leq |\mspec(M_0)|<\ell$. By the induction assumption, $i^*_0$ and $i^*_{01}$ are homotopy equivalences, with $j^*_0$ and $j^*_{01}$ as their respective homotopy inverses.

Since the maps $M\rightarrow M_0$, $M\rightarrow M_1$, $M_0\rightarrow M_{01}$ and $M_1\rightarrow M_{01}$ are canonical maps to a quotient, it follows that all  maps in the bottom plane are closed embeddings. As such, they are also cofibrations and hence, the pushout diagram is in particular a homotopy pushout. One easily sees that 
$$\kSpec(M)=\kSpec(M_0)\cup \kSpec(M_1), \quad \kSpec(M_{01})=\kSpec(M_0)\cap \kSpec(M_1).$$
Our next observation is that the quotient map $M\to M_0=M/(x)$ yields an isomorphism $M^{\x}\to M_0^{\x}$ and hence, $\kSpec(M_0^{\x})\rightarrow \kSpec(M^{\x})$ is a homeomorphism. The map $\kSpec(M_{01}^{\x})\rightarrow \kSpec(M_1^{\x})$ is also a homeomorphism, by the same reasoning. Thus, the top plane is also a homotopy pushout.

It follows from \cite{arkowitz}[Theorem 6.2.8] that the map $i^*:\kSpec(M)\rightarrow \kSpec(M^{\x})$ is a homotopy equivalence, if $i^*_{1}$ is. For this, we consider the following 2 cases:

\underline{Case 1}. Assume that there exist a pivotal element $x\in \Piv(M)$, for which $M^x$ is not a group. We can postulate in the above consideration that $y$ is a non-invertible element. One has $|\mspec(M_1)|<\ell$ by Lemma \ref{quot2} iii) and so, $i_1^*$ is a homotopy equivalence by the induction assumption. This implies that $i^*$ is also a homotopy equivalence as already observed, proving the induction step in this case.

\underline{Case 2}. It remains to consider the case, when for all $x\in \Piv(M)$, the monoid $M^x$ is a group. In this case, $|\mspec(M_1)|=\ell$ for any choice of $x$ and $y$. So we can't use the induction step to analyse $i_1^*$.

The key idea is to use the monoid $\sP(M)$ (see Section \ref{dis77}). It is a group in this case because $\sP(M)$ is a union of subgroups under our assumption. It is furthermore finitely generated since $\sP(M)$ is a subgroup of the group of invertible elements of $M$, which itself is the group of invertible elements of a finitely generated monoid $M$. Denote by $g_M$ the minimal number of group generators of $\sP(M)$. We will use a second induction on $g$.

Take $y$ to be a generator of $\sP(M)$ (thus, $y\neq 1$) and let $x\in \Piv(M)$, such that $y\in M^x$. The group $\sP(M_1)$ is a quotient of $\sP(M)$, where $M_1=M/(y\sim 1)$. This is by Lemma part iii) of \ref{d131.77}. Hence, $M_1$ is either separated (if $g_M=1$) or $g_{M_1}<g_M$. In both cases, $i_1^*$ is a homotopy equivalence. This finishes the proof.

\subsubsection{Integrality}

Finally, we want to remove our assumption that $M$ is integral from Proposition \ref{reductiontogroup}. This is done in a similar fashion to the above argument, but is a bit simpler.

We may assume that there are elements $x,y\in M$, with $x\neq 0, y\neq 0$ and $xy=0$. Further, $x\neq y$ since otherwise $x$ would be nilpotent. Denote
$$M_x:=M/(x), M_y:=M/(y) \text{ and } M_{xy}:=M/(x,y).$$
We will essentially repeat the first half of the above argument. Consider the following commutative diagram
$$\xymatrix{ & & M_y\ar[rd] &\\
	M\ar[rru]\ar[rd] & & & M_{xy}\\
	& M_x.\ar[rru] & & }$$
We have
$$\kSpec(M)=\kSpec(M_x)\cup \kSpec(M_y), \quad \kSpec(M_{xy})=\kSpec(M_x)\cap \kSpec(M_y).$$

We know that $M_x,M_y$ and $M_{xy}$ have less than $\ell$ prime ideals by Section \ref{1277prime} and hence, we can use the induction hypothesis. Combining that with part i) of Lemma \ref{quot1} shows that we have homotopy equivalences $\kSpec(M_x)\simeq \kSpec(M_x^\times)$, $\kSpec(M_y)\simeq \kSpec(M_y^\times)$ and $\kSpec(M_{xy})\simeq \kSpec(M_{xy}^\times)$. Using the fact that the induced maps on the $K$-spectrum are once again cofibrations as we did in Section \ref{sep} yields the proof of Theorem \ref{reductiontogroupsnonsep}.

The next example shows that even if $M$ has no idempotents, the can appear when one considers the admissible decomposition:

\begin{Exm} Let $M=\gen{x,y,z}/(xy=x, yz^2=yz)$. It is obvious that $M$ has no idempotents. However, $M_0=M/x=\gen{y,z}/yz^2=yz$ and $M_1=\gen{x,z}/z^2=z$. We see that $M_1$ has $z$ as an idempotent.
\end{Exm}

\begin{Exm}\label{idem} Let $M=\gen{x,y}/x^2y=x$. It has the idempotent $(xy)^2=xy$, and clearly no other. We have $\Adm(M)=\{(0), (xy)\}$. This yields
$$\kSpec(M)\simeq\kSpec(M/xy\sim 1))\sqcup\kSpec(M/x)\simeq\kSpec(\Z^\circ)\sqcup\kSpec(\N^\circ),$$
where $\Z^\circ$ is the free grouplike binoid with one generator and $\N^\circ$ if the free binoid with one generator.
\end{Exm}

As a consequence, we can now prove the following result, which is closely related to \cite[Corollary 4.5.8]{S.Bot}, where the author proves this result for hypersurfaces, but for characteristic zero fields.

\begin{Th}\label{complexcomponents} Let $M\neq 0$ be a binoid and denote by $K_M$ the number of connected components of the binoid algebra $\C[M]$. We have
$$K_M=\sum_{\r\in\Adm}\mathsf{Tors}(M(\r)^\times),$$
where $\mathsf{Tors}(M(\r)^\times)$ denotes the torsion elements in the subgroup of invertible elements of $M(\r)$.
\end{Th}

\begin{proof} For any ring $R$, the non-existence of idempotent elements is equivalent to the connectedness of $\Spec(R)$ in the Zariski topology. For a $\C$ algebra, this is equivalent to the connectedness of $\C \Spec(R)$ with its complex topology. By Theorem \ref{reductiontogroupsnonsep}, we have
$$\C \Spec(\C[M])\simeq\C \Spec(M)\simeq\coprod\limits_{\r\in \Adm(M)} \C\Spec((M(\r)^{\x}).$$
For every $\r\in \Adm(M)$, we have
$$\C\Spec((M(\r)^{\x})\simeq\C\Spec((\Z^\circ)^{n_\r}\wedge \mathsf{Tors}(M(\r))^{\x})\simeq\C\Spec(\Z^\circ)^{n_\r}\times\C \Spec( \mathsf{Tors}(M(\r)^{\x})).$$
Clearly $\C \Spec[\Z^\circ]$ is connected and so are all of its powers. On the other hand, as $\C$ is algebraically closed, the connected components of $\C\Spec(\mathsf{Tors}(M(\r)^{\x}))$ agree with the elements of $\mathsf{Tors}(M(\r))$.
\end{proof}

\begin{Co} Let $M\neq 0$ be a binoid with no non-trivial idempotent elements and with a torsionfree group of units. Let $\k$ denote a field of characteristic $0$. Then the binoid algebra $\k[M]$ has no non-trivial idempotent elements and $\Spec \k[M]$ is connected.
\end{Co}

\begin{proof} 
We can put all relevant data of $K$ in a subfield of $\C$. Hence this is a special case of Theorem \ref{complexcomponents}.
\end{proof}

\section{Further simplifications}

We have reduced the study of the homotopy of $\kSpec(M)$ for finitely generated binoids, to the disjoint union of $\kSpec(G_i)$-s, where the $G_i$-s are grouplike binoids. Denote the free grouplike binoid with 1 generator by $\Z^\circ$ and the cyclic binoid with n+1 elements (the additional element being the absorbing element) by $C_n^\circ:=\Z^\circ/n\Z^\circ=\gen{x}/x^n=1$.

Let $G$ be a grouplike binoid. We define by $C_2^\circ\otimes G$ the grouplike binoid obtained by imposing $g^2=1$ for all $g\neq 0\in G$.

\begin{Pro}\label{caseofgroups} Let $G$ be a grouplike binoid. We have a homotopy equivalence
$$\RSpec(C_2^\circ\otimes G)\rightarrow \RSpec(G).$$
\end{Pro}

\begin{proof} Let 
$$G=(\Z^\circ)^n\wedge C_{e_1}^\circ\wedge\cdots\wedge C_{e_l}^\circ\wedge C_{o_1}^\circ\wedge\cdots\wedge C_{o_m}^\circ$$
be a decomposition of $G$, where the $e_i$-s are even numbers and $o_j$-s are odd numbers. First, observe that since $\kSpec(M)$ is defined to be $\Hom_\mathsf{bn}(M,\k)$, it respects products. As such, we have
\begin{equation*}
\resizebox{0,975\hsize}{!}{$\RSpec(G)\simeq\RSpec(\Z^\circ)^n\times \RSpec(C_{e_1}^\circ)\times\cdots\times \RSpec(C_{e_l}^\circ)\times \RSpec(C_{o_1}^\circ)\times\cdots\times \RSpec(C_{o_m}^\circ)$.}
\end{equation*}

We immediately see that $\RSpec(\Z^\circ)$ is homotopy equivalent to $\RSpec(C_2^\circ)$, which takes care of the free part. The cyclic parts are even easier, since $\RSpec(C_n^\circ)$ is the set of solutions to $x^n=1$ over the real numbers. Hence,
\[ \RSpec(C_n^\circ)=\begin{cases}
\RSpec(C_2^\circ) \ \textnormal{ if } \ n=\textnormal{even}, \\
\RSpec(C_1^\circ) \ \textnormal{ if } \ n=\textnormal{odd}.
\end{cases}
\]
As such, we can assume that, up to homotopy, $\RSpec(G)$ is the collection 
$$\{(\underbrace{\pm 1,\ldots,\pm 1}_{\times n},\underbrace{\pm 1,\ldots,\pm 1}_{\times l},\underbrace{1,\ldots,1}_{\times m})\},$$
implying our result.
\end{proof}

We point out that the homotopy inverse does not come from a binoid homomorphism.

\begin{Co}\label{pi0group} Let $M$ be a binoid. There is a homotopy equivalence
$$\RSpec(M)\simeq \coprod\limits_{\r\in \Adm(M)}\RSpec((C_2^\circ)^{n_\r})\simeq \coprod\limits_{\r\in \Adm(M)}(\RSpec(C_2^\circ))^{n_\r},$$
where $n_\r$ is the number of minimal generators of $M(\r)^{\x}\otimes C_2^\circ= M(\r)^{\x}/m^2\sim 1$. In particular,
$$\Pi_1(\RSpec(M))\simeq \coprod_{\r\in\Adm(M)}\cD^{n_\r},$$
where $\cD^{n_\r}$ is the discrete groupoid with $2^{n_\r}$ objects. The coproduct is taken in the 2-category of groupoids, and is equivalent to their 2-coproduct.
\end{Co}

\begin{proof} This is a direct result of Theorem \ref{reductiontogroupsnonsep} and Proposition \ref{caseofgroups}. The last assertion (the equivalence of the coproduct and the 2-coproduct) will be shown in Corollary \ref{coprod=2-coprod}.
\end{proof}

It follows directly from the corollary that
$$\pi_0(\RSpec(M))\simeq \bigcupplus_{\r\in \Adm(M)}C_2^{n_\r},$$
where $\bigcupplus$ denotes the disjoint union. Any point $\px\in\pi_0(\RSpec(M))$ can be seen as
$$\px=((-1)^{\delta_1},(-1)^{\delta_2}\ldots,(-1)^{\delta_{n_r}}), \ \ \delta_i\in\{0,1\},$$
for some $r\in\Adm(M)$.

\begin{Exm} Let us consider the binoid given by $M=\gen{x,y,y^{-1},z}/x^2y=z^2$. It is easily seen that $M$ has no non-trivial idempotents. As such,
$$\RSpec(M)\simeq\RSpec(C_2^\circ).$$
\end{Exm}

\subsection{Morphisms}\label{morps}

Let $\alpha:M\rightarrow N$ be a binoid homomorphism. This induces a map
$$\alpha_*:\pi_0(\RSpec(N))\rightarrow \pi_0(\RSpec(M)).$$
Due to the above corollary, we may rewrite it as
$$\alpha_*:\bigcupplus_{\r\in \Adm(N)}(C_2^\circ)^{n_\r}\rightarrow \bigcupplus_{\t\in \Adm(M)}(C_2^\circ)^{m_\t}.$$
Let $\px$ be a point of $\pi_0(\RSpec(N))\simeq \bigcupplus_{\r\in \Adm(N)}(C_2^\circ)^{n_\r}$. Since its an element of a disjoint union of sets, it is inside a single one. Hence, we may simply write $\px\in(C_2^\circ)^n$. Likewise, for it's image, we may write $\alpha_*(\px)\in(C_2^\circ)^m$.

Let $G=(C_2^\circ)^n$ and $H=(C_2^\circ)^m$, with $\{x_1,\ldots,x_n\}$ and $\{y_1,\ldots,y_m\}$ as their respective minimal generators. We have $\alpha_*(x_i)=\prod_{j=1}^my_j^{\delta_{ij}}, \delta_{ij}=0,1$. Denote by $\Supp_\alpha(x_i)$ the subset of $\{1,\ldots,m\}$, for which $\delta_{ij}=1$. We rewrite $\alpha_*(x_i)=\prod\limits_{j\in\Supp_\alpha(x_i)}y_j$.
We have
\begin{eqnarray*} \alpha_*(\px)&=&(\prod_{j\in\Supp_\alpha(x_1)}(-1)^{\delta_j},\prod_{j\in\Supp_\alpha(x_2)}(-1)^{\delta_j},\ldots,\prod_{j\in\Supp_\alpha(x_n)}(-1)^{\delta_j})\\
		  &=&((-1)^{\sum\limits_{j\in\Supp_\alpha(x_1)}\delta_j},(-1)^{\sum\limits_{j\in\Supp_\alpha(x_2)}\delta_j},\ldots,(-1)^{\sum\limits_{j\in\Supp_\alpha(x_n)}\delta_j}).
\end{eqnarray*}

For every $1\leq i\leq n$, denote by $k_i$ the number of elements in $\Supp_\alpha(x_i)$. We can reduce the above formula to the following:

\begin{eqnarray}\label{mapingformula} \alpha(\px)=((-1)^{k_1},(-1)^{k_2},\ldots,(-1)^{k_n}).
\end{eqnarray}

\begin{Co} Let $\alpha:M\rightarrow N$ a homomorphism of binoids with no non-trivial idempotents. The induced map
$$\alpha_*:\pi_0(\RSpec(N))\rightarrow \pi_0(\RSpec(M))$$
is a group homomorphism.
\end{Co}

\begin{Rem} This immediately implies that for binoids without non-trivial idempotents, the fibre of every $\py\in\pi_0(\RSpec(M))$ has the same number of elements. As such, in order to calculate the number of elements in the image of $\alpha_*$, it suffices to calculate the number of elements mapping to the unit $(1,1,\ldots,1)$. That is to say, the collections of $\{m_i\}$ (as given in Formula (\ref{mapingformula})), for which $m_i$ is even for all $1\leq i\leq n$.
\end{Rem}

\section{Schemes} \label{schemes}

Let us recall the definition of a binoid scheme. Let $M$ be a commutative binoid and $\mspec(M)$ the set of prime ideals of $M$. Take an element $f\in M$. We define the subset $D(f):=\{\p\in\mspec(M) | f\not \in \p\}$. Since $D(f)\cap D(g)= D(fg)$, the collection of all $D(f), f\in M$ is an open base of topology. This is usually called the \emph{Zariski topology}. The association
$$D(f)\mapsto M_f$$
defines a sheaf of binoids, called the \emph{structure sheaf} of $M$. The topological space $\mspec(M)$, together with its structure sheaf is called an \emph{affine binoid scheme}, and is written as $(\mspec(M),\mathcal{O}_{\mspec(M)})$, or simply $\mspec(M)$ when there is no ambiguity. A binoid scheme $(X,\mathcal{O}_X)$ is a sheaf of binoids that locally looks like an affine binoid scheme. 

\subsection{Realisation of binoid schemes}

Clearly, if $M_f\simeq N_g$ are isomorphic binoids, we have $\kSpec(M_f)\simeq \kSpec(N_g)$. Hence, we can talk about $\k X$, for a binoid scheme $X$. It follows immediately that an affine cover $\U=\{U_i\}=\{\mspec(M_i)\}$ gives rise to an open cover of $\k X$.

We say that a poset $(P,\leq)$ is \emph{locally a lattice}, if for every $p\in P$, $L(P,p):=\{q\in P|q\leq p\}$ is a lattice. We know that the underlying topological space of $T(X)$ of a binoid scheme $X$ is locally a lattice poset \cite{chww}, \cite{spec}. This is called the \emph{underlying poset} of $X$.

In the affine case, the elements of this poset (which is a (join) lattice) correspond to prime ideals and $p\leq q$ if and only if $p\subseteq q$. We note that in the finitely generated case, we can replace prime ideals with elements from $M^{\sf sl}$, see Subsection \ref{msl}.

A binoid scheme $X$ is said to be \emph{quasi-separated} if the intersection of affine subschemes is again affine. One can see that this simply means that the connected components of the underlying topological poset of $X$ are globally meet-semilattices with a maximal element \cite[Corollary 3.7]{chww}. We remark that there is a small clash of terminology. The authors in the above paper work with monoids with an absorbing element. However, they do not require for a homomorphism to preserve it. Furthermore, their use of the word separated is not to be confused with our notion of a separated binoid. It also differs from our notion of quasi-separatedness, though their notion of a separated binoid scheme implies that it is quasi-separated.

Note that the underlying meet semi-lattice of $X$ need not have a unit in general, even if it is connected. Indeed, it will have a unit if and only if it is affine.

A sheaf of binoids over $X$ becomes simply a contravariant functor
$$\mathcal{S}:T(X)\rightarrow \mathsf{binoids},$$
where we look at the poset $T(X)$ as a category in the natural way, see \cite[Proposition 2.11]{chww}.

Henceforth, unless otherwise stated, a binoid scheme $X$ is assumed to be of finite type and quasi-separated.

Let us fix such an affine covering $\{M_i\},i\in\{1,\ldots, k\}$, of $X$. We aim to give a method of calculating all the homology groups, as well as the fundamental groupoid $\Pi_1(\R X)$ of the 'real' geometric object $\R X$ associated to $X$.

\subsection{Homology of $\R X$}

\subsubsection{Preliminary discussion}

Recall that a \emph{simplicial complex} is a set $K$, together with a collection of finite non-empty subsets of $K$ called \emph{faces} or \emph{simplices}, such that any singleton is a face (or simplex) and every non-empty subset of a face is a face. A face $\sigma$ is of dimension $n$, or simply $n$-face, if $|\sigma|=n+1$. The collection of $n$-faces is denoted by $K_n$. We assume that there is given a total order on $K$. For a  $k$-face $\sigma=\{v_0<v_1<\cdots <v_k\}$, $v_i\in K$ and any integer $i\in \{0,1,\cdots ,k\}$, one denotes by $\partial_i(\sigma)$ the $(k-1)$-dimensional face
$$\{v_0<\cdots<v_{i-1}<v_{i+1}<\cdots<v_k\}.$$

The natural ordering of the power set of $K$ induces an ordering on the set of all faces. Hence, one can consider contravariant functors defined on the set of all faces. Any such functor $A$ assigns an abelian group $A(\sigma)$ to a face $\sigma$. Moreover, if $\tau\subseteq \sigma$, there is a well-defined map $\alpha^{\sigma,\tau}:A(\sigma)\to A(\tau)$, such that $\alpha^{\sigma,\sigma}=\id_{A(\sigma)}$, and for  $\rho\subseteq \tau\subseteq \sigma$, one has $\alpha^{\sigma,\rho}=\alpha^{\tau,\rho}\circ\alpha^{\sigma,\tau}$. 

One can define a chain complex $C_*(K,A)$ as follows: The group of $n$-chains is  given by
$$C_n(K,A)=\bigoplus_{\sigma \in K_n}A(\sigma).$$
The inclusion map $A(\sigma)\to C_n(K,A)$ is denoted by $j^\sigma$. The  boundary map $$d_n:C_{n}(K,A)\to C_{n-1}(K,A)$$ 
is given by
$$d_n( j^{\sigma}(a)) =\sum_{i=0}^n(-1)^i j^{\partial_i \sigma}(\alpha^{\sigma,\partial_i \sigma}(a)).$$ 
Here, $\sigma\in K_n$ and $a\in A(\sigma)$. The homology of the chain complex $C_*(K,A)$ is denoted by $H_*(K,A)$. If $A$ is the constant functor with values in $\mathbb{Z}$, these groups coincides with the classical groups $H_*(K,\mathbb{Z})$. 

These groups appear in the classical spectral sequence of the homology of a covering (which is a homological version of \cite[Section II. 2.5]{R. Godement}). Let $S$ be a topological space and $\mathcal{U}=\{U_i\}_{i\in I}$ an open covering of $S$. We assume that $I$ is a totally ordered set. For $\sigma=(i_0<i_1<\cdots <i_k)$ , we set
$$U_\sigma=U_{i_0\cdots i_k}=U_{i_0}\cap \cdots \cap U_{i_k}.$$
Recall that the  nerve $\mathcal{N}$ of $\mathcal{U}=\{U_i\}_{i\in I}$ is  the simplicial complex,  whose vertices are elements of $I$. Moreover, $\sigma$ is a $k$-simplex of $\mathcal N$ if $U_{\sigma}\not= \emptyset$. We denote by ${\mathcal N}_k$ the set of all $k$-simplexes of $\mathcal N$.

If $\tau\subseteq \sigma$, then $U_\sigma\subseteq U_\tau$. Thus, by the functoriality of the singular homology with values in a group $G$ of the topological space, one obtains an induced homomorphism $H_q(U_\sigma)\to H_q(U_\tau)$, $q\geq 0$. This allows us to consider the homology groups $H_*(\mathcal{N},H_q)$. These are homology groups of the chain complex with values in $\mathbb{Z}$, whose $p$-th component is given by
$$\bigoplus_{(i_0<i_1<\cdots <i_p)\in {\mathcal N}_p} H_q(U_{i_0\cdots i_p}).$$
It is well-known that there exists a first quadrant homological spectral sequence for which
$$E^2_{pq}=H_p(\mathcal{N},H_q)\Longrightarrow H_{p+q}(S).$$

\subsubsection{Calculating the homology of $\R X$}\label{homology}

Let $X$ be a quasi-separated binoid scheme. Consider its open and affine covering $\{\mspec(M)_i\}_{i\in\{1,\ldots ,k\}}$. For the corresponding topological space $\R X$, we have an open cover given by 
$U_i=\RSpec(M_i)$. It follows from Corollary 3.1.2 and the well known identity $H_*(X\sqcup Y)\simeq H_*(X)\oplus H_*(Y)$, that $H_q(\RSpec(M_i))=0$ for all $q\geq 1$. Here, $H_*(-)$ is used to denote $H_*(-,\Z)$. As such, $E^2_{pq}=0$ for $q\geq 1$. The spectral sequence degenerates to an isomorphism
$$H_*(\R X)=H_*(\mathcal{N}, H_0).$$
In other words, these groups are isomorphic to the homology of the following cochain complex
$$\bigoplus_{i}H_0(U_i)\longleftarrow \cdots \longleftarrow \bigoplus_{i_0<\cdots< i_k}H_0(U_{i_0\cdots i_k})\longleftarrow\cdots .$$ 
Here, the sum is taken over all $i_0<\cdots< i_k$ for which $U_{i_0\cdots i_k}\not =\emptyset$.

\subsection{The Fundamental Groupoid $\Pi_1(\R X)$:}

We will give a brief description of how to calculate the fundamental groupoid of $\R X$. As one might expect, this will require several results of 2-category theory. We will, however, try to keep things as simple as possible, and will state many results without proofs, or even strict definitions. The interested reader will find the necessary results in \cite{p3}.

\subsubsection{Prerequisite results of 2-category theory:}

Recall that a \emph{2-category} $\fC$ is a category enriched in categories. In other words, it has objects, often called \emph{0-cells}, and for every two objects $\mathcal{A},\mathcal{B}$, their morphisms $\Hom_\fC(\mathcal{A},\mathcal{B})$ forms a category. The objects of this category are usually called \emph{1-cells} or \emph{morphisms} and the morphisms of $\Hom_\fC(\mathcal{A},\mathcal{B})$ are called \emph{2-cells} or \emph{2-morphisms}. There are a few other technical conditions, which we will omit.

The prime example of a 2-category is the 2-category of small categories. Here, objects (0-cells) are small categories, morphisms (1-cells) are functors and 2-morphisms (2-cells) are natural transformations.

A \emph{2-functor} (also called pseudo-functors) $\F:\fC\rightarrow \mathfrak{D}$ takes objects to objects, morphisms to morphisms and 2-morphisms to 2-morphisms. Note, however,  that compositions need to go to compositions. Instead, for any $\mathcal{A}\xrightarrow{\mathcal{F}}\mathcal{B}\xrightarrow{\mathcal{G}}\mathcal{C}$, there exists a natural isomorphism $\alpha_{\mathcal{F},\mathcal{G}}:\F(\mathcal{F})\circ\F(\mathcal{G})\Rightarrow\F(\mathcal{F}\circ \mathcal{G})$, which satisfies a certain compatibility condition for any $\mathcal{A}\xrightarrow{\mathcal{F}}\mathcal{B}\xrightarrow{\mathcal{G}}\mathcal{C}\xrightarrow{\mathcal{H}}\mathcal{D}$.

We can also talk of morphisms between 2-functors, but we will not discuss them in this paper. If the natural isomorphism $\alpha_{\mathcal{F},\mathcal{G}}$ is the identity, the 2-functor is said to be \emph{strict}.

Any category $\mathcal{C}$ can be seen as a 2-category, where the 2-morphisms are just the identities. Similar to how a set can be viewed as a category. Likewise, a functor $\mathcal{F}:\mathcal{C}\rightarrow\mathcal{D}$ can be viewed as a 2-functor. Moreover, such 2-functors will always be strict. 

For any 2-functor $\F:\fC\rightarrow\mathfrak{D}$, we can consider its 2-limit and 2-colimit. These are objects in $\mathfrak{D}$ and defined much like classical limits and colimits, except that the diagrams only need to \emph{2-commute}. That is to say, commute up to a natural isomorphism. This natural isomorphism itself needs to satisfy a certain 2-commutative diagram. A 2-(co)limit is unique up to an equivalence.

For strict 2-functors, we can consider both 2-limits and 2-colimits, as well as limits and colimits. These are usually not equivalent. A (co)limit is unique up to an isomorphism. Though 2-(co)limits are significantly harder to work with than (co)limits, one calculatory advantage that they hold is that in a 2-(co)limit, we may replace the participating objects (and morphisms) with equivalent ones. In a (co)limit, we may only replace them with isomorphic ones.

Let $\mathcal{F}:\cC\rightarrow\cD$ be a functor. It is said to be an \emph{equivalence}, if there exists a functor $\cG:\cD\rightarrow\cC$, such that $\cG\circ \cF$ and $\cF\circ\cG$ are isomorphic to the identities. The functor $\cF$ is said to be an \emph{isomorphism} if $\cG\circ\cF$ and $\cF\circ\cG$ are equal to the identities.

It is well known that a functor $\cF:\cC\rightarrow\cD$ is an equivalence if it is
\begin{itemize}
	\item \emph{full and faithful}, that is for all $C,C'\in\cC$, we have a bijection
	$$\Hom_\cC(C,C')\simeq\Hom_\cD(\cF(C),\cF(C'))$$
	and
	\item \emph{essentially surjective}, that is for all $D\in\cD$, there exist $D'\in\cD$ and $C\in \cC$, such that $\cF(C)=D'$, with $D'\simeq D$ being isomorphic.	
\end{itemize}
Likewise, $\cF$ is an isomorphism if it is full and faithful and bijective on objects. By the latter we mean that for all $D\in\cD$, there exists a unique object $C\in\cC$ with $\cF(C)=D$.

A functor $\cF:\cC\rightarrow \cD$ is said to be injective on objects if $\cF(A)=\cF(B)$ implies $A=B$. Note that, we do mean equality, not isomorphism. This condition will play an essential role in reducing the 2-colimit to the far simpler colimit.

Let $\cF:\cG\rightarrow\mathcal{H}$ be a functor. Assume there exist $x\neq x'\in\cG$, such that $\cF(x)=\cF(x')=y$. We define $\mathcal{H}'$ to be the category generated by adding a new object $y'$ to $\mathcal{H}$ and an isomorphism $\alpha:y\rightarrow y'$. Let $z,z'$ be objects of $\mathcal{H}'$. We have
$$\Hom_{\mathcal{H}'}(z,z')=
\begin{cases} \Hom_{\mathcal{H}}(z,z'), & z\not =y'\not =z'\\
			  \Hom_{\mathcal{H}}(y,z')\alpha^{-1}, & z =y'\not =z'\\
			  \alpha \Hom_{\mathcal{H}}(z,y), & z \not=y' =z'\\
			  \alpha \Hom_{\mathcal{H}}(y,y)\alpha^{-1}, & z =y'=z'.
\end{cases}$$
We also define a new functor $\cF':\cG\rightarrow \mathcal{H}'$ as follows. On objects, we set
$$\cF'(u)=
\begin{cases} \cF(u), & u\not =x'\\
			  y,& u=x'.
\end{cases}$$
Here, $u$ is an object of $\mathcal{H}$. On morphisms, $\cF'$ is defined by
$$\cF'(u\xto{\beta}v)=
\begin{cases} \cF(\beta), & u\not =x'\not =v\\
			  \cF(\beta)\alpha^{-1},& u=x'\not = v'\\
			  \alpha \cF(\beta),& u\not =x'= v'\\
			  \alpha \cF(\beta)\alpha^{-1},& u=x' = v'.
\end{cases}$$
The two functors $\cF$ and $\cF'$ are isomorphic via $\theta:\cF\to cF'$, where
$$\theta(u)=
\begin{cases} Id_{\cF(u)}, & u\not =x'\\
			  \alpha,& u=x'.
\end{cases}$$
We call this construction \emph{stretching}\label{stretching}.

\subsubsection{Groupoids}

A \emph{groupoid} is a category, where every morphism is an isomorphism. There is a strong connection between groups and groupoids. Let $G$ be a group. We can create a groupoid $\underline{G}$, which has only a single object $x\in\underline{G}$, with $\mathsf{Aut}(x)=G$. A groupoid which is equivalent to a groupoid with only a single object is said to be \emph{connected}.

Like with any category, if we take a groupoid $\cG$ and take its skeleton (that is we choose representatives in isomorphic classes of objects), we will get an equivalent (but not isomorphic) groupoid $\cG'$. In this way, we can think of a groupoid as a \emph{disjoint collection of groups}. More precisely, every groupoid $\cG$ is equivalent to the 2-coproduct $\coprod_i \underline{G_i}$ of groupoids coming from groups.

Reversely, we can also add isomorphic objects to our groupoid without fundamentally changing it. This is useful in making a functor injective on objects. 

A groupoid $\cD$ is said to be \emph{discrete}, if for every object $x\in\cD$, we have $\mathsf{Aut}(x)=\{1\}$, the trivial group with one object. These groupoids can be thought of as coming from sets. These types of groupoids will play a vital part in calculating the fundamental groupoid (defined below) $\Pi_1(\R X)$. 

We will need the classical construction of colimits of groupoids, which can be found in \cite{gz} (see also \cite{p7}). According to \cite[pp.10-11]{gz}, the colimit of a diagram of groupoids can be computed as the colimit in the category of small categories. The description of the later can be found in \cite[pp. 4-5]{gz}.

We are interested in the fundamental groupoid up to an equivalence of categories. As such, there is one more simplification we can make. In the colimit, before we span a category with it, but after we have taken the colimit of the objects and of the morphisms, we equate two objects if there exists at least one connecting isomorphism between them. We then have to set exactly one connecting isomorphism to the identity. If there are more than one connecting isomorphisms, it does not matter which one we choose. The others will then become automorphisms.

\subsubsection{The fundamental groupoid}

Our interest in groupoids comes from the \emph{fundamental groupoid} $\Pi_1(X)$ of a topological space $X$. It is defined to be a groupoid whose objects are the points of $X$. For any two objects $x,y\in\Pi_1(X)$, we define $\Hom_{\Pi_1}(x,y)$ to be the homotopy classes of paths from the point $x\in X$ to the point $y\in X$. It is clear that this defines a groupoid, since if there exists a path $\gamma:x\rightarrow y$, there exists a path $\gamma^{-1}:y\rightarrow x$. We immediately see that $\mathsf{Aut}(x)\simeq \pi_1(X,x)$ by definition. Further, if the points $x$ and $y$ are path-connected in $X$, the objects $x$ and $y$ are isomorphic in $\Pi_1(X)$. Hence, the isomorphism classes of $\Pi_1(X)$ gives us $\pi_0(X)$. Indeed, the fundamental groupoid can be seen as `the collection of fundamental groups at every connected component'. But, $\Pi_1(X)$ has an other advantage over $\pi_1(X,x)$ than just generality: The fundamental groupoid is not dependant on a basepoint $x\in X$. As such, it has far nicer categorical properties. For example, the Seifert-van Kampen theorem for fundamental groupoids does not require a common point \cite{brown} and works for every  open covering, which is something we will heavily exploit in the upcoming section. An other advantage, which will not play a role in this paper, is that the fundamental groupoid can be defined categorically (or axiomatised), as shown in \cite{p3}.

\subsubsection{Calculating $\Pi_1(\R X)$}

The following results will enable us to calculate the fundamental groupoid of a real binoid scheme. In order to state them, however, we need to introduce a few notations:

Let $I$ be a set. Denote by $\fb(I)$ the poset of all proper subsets. For a natural number $n\geq 1$, we write ${\bf n}:=\{1,2,\cdots,n\}$. Accordingly, we write $\fb({\bf n})$ instead of $\fb(\{1,\cdots,n\})$. For  a subset $S\subseteq {\bf n}$, we write $S^\c$ instead  of ${\bf n} \setminus S$. Let $J\subsetneq I$ be a proper subset. We set
$$\fb^\c(I:J)=\{X^\c | X\in \fb(I) \textnormal{ and } J\subseteq X\}.$$
That is
$$\fb^\c(I:J)=\{Y| I^\c\subsetneq Y\subseteq J^\c\}.$$
For a proper subset of $I \in\fb({\bf n})$, we will say that a 2-functor $\Phi:\fb({\bf n})^{\mathsf op}\rightarrow \Grpd$ satisfies the condition $B^I_J$ if the canonical functor
$${\sf colim}_{\fb^\c(I:J)} \Phi \to \Phi(I^\c)$$
is injective on objects. Here, $\Phi(I^\c)$ denotes the complement of $I$, as a subset of ${\bf n}$, i.e., $I^\c:={\bf n}\setminus I$.

For a topological space $X$ and a covering $\mathcal{U}=\{U_i\}_{i\in I}$, we use the notation $\hat{C}_{\mathcal{U}}(X)$ for the associated \^Cech complex. This is a natural poset (indeed a lattice) under inclusion. From here on onwards, we will identify $U_\alpha\in \hat{C}_{\mathcal{U}}(X)$ with its index set $\alpha$. Note that, by $U_\alpha$ we mean the element of the \^Cech covering of $\mathcal{U}$, not the corresponding open subset. The distinction is that, if $\alpha\neq \beta$, $U_\aleph\neq U_\beta$ as elements of the poset. However, as open subsets, it could be that $U_\alpha=U_\beta$. This way, we may use the notations introduced above.

We can now state our main theorems.

\begin{Th}[\cite{p3}, Thm. 3.2]\label{Seifert-vanKampen} Let $X$ be a topological space and $\mathcal{U}=\{U_i\}_{i\in I}$ an open covering. There exists an equivalence of categories
$$\Pi_1(X)\simeq\tcolim_{\hat{C}_{\mathcal{U}}(X)}\Pi_1.$$
\end{Th}
\begin{Th}[\cite{p7}, Cor. 5.4]\label{colim=2-colim} Let $X$ be a topological space, $I={\bf n}$ a finite set, $\mathcal{U}:=\{U_i\}_{i\in I}$ an open covering of $X$ and $\mathfrak{F}:\hat{C}_{\mathcal{U}}(X)\rightarrow \mathsf{Groupoids}$ a strict 2-functor. The natural functor
$$\colim_{\hat{C}_{\mathcal{U}}(X)}\mathfrak{F}\rightarrow\tcolim_{\hat{C}_{\mathcal{U}}(X)}\F$$
is an equivalence of categories if the following conditions hold:
\begin{itemize}
	\item[B1)] The condition $B_\emptyset^{\bf k}$ holds for any integer $k=n-1,n-2$.
	\item[B2)] Let $1\leq k\leq n-2$ be an integer and $J\subseteq \{k+2,\ldots,n\}$ and set with $|J|\geq n-k-2$. The condition $B_J^I$ holds, where $I:={\bf k}\sqcup J$.
\end{itemize}
\end{Th}

\begin{proof} The proof of this is given in \cite[Cor 5.4]{p7}. However, in this paper, we used the dual of the poset
$\fb^\c(I:J)$ (duality is given by $X\mapsto X^\c$) and covariant functors. Since our functors are contravariant we can apply that result of \cite{p7}.
\end{proof}

\begin{Rem} We can always manipulate our 2-functor $\F$ in such a way that the functors ${\sf colim}_{\fb^\c(I:J)} \Phi \to \Phi(I^\c)$ become injective on objects. This follows from our discussion on stretching on page \pageref{stretching}. 
\end{Rem}
\begin{Rem}
The conditions of this theorem can be simplified if the functors $\F(U_\alpha)\rightarrow \F(U_\beta), \alpha\leq \beta$ are surjective on objects. In such a case, we only need to ensure that the functors $\F(U_\alpha)\rightarrow \F(U_\beta)$ are injective on objects. This will imply bijectivity on objects of the functors $\F(U_\alpha)\rightarrow \F(U_\beta)$, and as such, ${\sf colim}_{\fb^\c(I:J)} \Phi \to \Phi(I^\c)$ will also be bijective on objects.
\end{Rem}

As a direct consequence of Theorem \ref{colim=2-colim}, we get this corollary:

\begin{Co}\label{coprod=2-coprod} Let $I$ be a finite set and $\{\cG_i\}_{i\in I}$ a groupoid for every $i\in I$. The coproduct and 2-coproduct of the $\cG_i$-s is equivalent. Hence, we can use the common notation $\coprod_{i\in I}\cG_i$ to mean both the coproduct and the 2-coproduct.
\end{Co}

\begin{proof} Let $I={\bf n}$ and consider the 2-functor $\F:\fb({\bf n})\rightarrow \mathsf{Groupoids}$, given as follows: We set $i\mapsto \cG_i$, $i\in I$. For all $S\subseteq\fb({\bf n})$ with $|S|\geq 2$, we set $\F(S)=\emptyset$. The result now follows.
\end{proof}

\subsubsection{Calculating the fundamental groupoid of $\R X$}

Let $X$ be a quasi-separated binoid scheme, with an affine covering $\A:=\{\mspec(M_i)\}_{i\in\{1,\ldots ,k\}}$. Our aim is to calculate $\Pi_1(\R X)$. For this, consider $\hat{C}_\A(X)$, the underlying poset of the \^Cech complex associated to the affine covering of $X$. The element of $\hat{C}_\A(X)$ corresponding to $\mspec(M_i)$ will simply be denoted by $i$.

We define a functor from the poset $\hat{C}_\A(X)$ to the category of topological spaces
$$\cal{R}:\hat{C}_\A(X)\rightarrow \mathsf{Top},$$
given by
$$i\mapsto \coprod_{r\in\Adm(M_i)} \RSpec(M_i(\r)^{\x}\otimes C_2^\circ).$$
Functoriality of this association is clear, since for all $ij\rightarrow i$, we get $M_i\rightarrow M_{ij}$, hence $\RSpec(M_{ij})\rightarrow\RSpec(M_i)$. Corollary \ref{pi0group} now gives us the desired result.

Using this, we define a strict 2-functor
$$\fR:\hat{C}_\A\rightarrow \mathsf{Grpd}$$
with values in the 2-category of groupoids, given by
$$i\mapsto \Pi_1\left(\coprod_{r\in\Adm(M_i)} \RSpec(M_i(\r)^{\x}\otimes (C_2^\circ)^{n_\r})\right)\simeq\coprod_{\r\in\Adm(M_i)}\cD^{n_\r}.$$
Theorem \ref{Seifert-vanKampen} and the fact that we can replace equivalent groupoids with equivalent ones in a 2-colimit, says that the fundamental groupoid of $\R X$ is equivalent to the 2-colimit of $\fR$. Since calculating the colimit is significantly easier than calculating the 2-colimit, we wish to use Theorem \ref{colim=2-colim}. Our discussion regarding the stretching of functors on page \pageref{stretching} allows us to do just that.

\section{Applications}\label{subsec.app}

We aim to give a practical demonstration on how to use the above results to calculate the homology groups, as well as the fundamental groupoid of $\R X$, where $X$ is a binoid scheme. We will do so both for a general class (see the Stanley-Reisner case) as well as a specific example.

We point out, that if we glue binoid schemes on affine subsets, the obtained binoid scheme will be quasi-separated.

\subsection{Punctured spectra of Stanley-Reisner rings}

Let $M$ be a binoid. Since the union of prime ideals is prime, it has a unique maximal prime ideal, denoted $\m_M$ or simply $\m$. It consists of the set of non-invertible elements and is the unique closed point of $\mspec(M)$. Its complement is denoted by $\mspec^\bullet(M):=\mspec(M)\setminus\m_M$. By restricting the structure sheaf $\mathcal{O}_{\mspec(M)}$ to $\mspec^\bullet(M)$, we get an open subscheme, called the \emph{punctured spectrum} of $M$.

A simplicial complex $\Delta$ on $ {\bf n}=\{1, \ldots, n\} $ defines the (multiplicatively written) binoid \[ M_\Delta=\{ X_1^{r_1} \cdots X_n^{r_n}  ,\, \prod_{i \in I} X_i =0 \text{ for nonfaces } I  \} \] and its binoid algebra is the \emph{Stanley-Reisner Ring of $\Delta$}. We want to compute the fundamental groupoid of the punctured spectrum of the Stanley-Reisner ring with our methods. In other words, we want to calculate $\Pi_1(\RSpec^\bullet(M_\Delta))$, or $\R^\bullet[\Delta]$ for short.

It is clear that there exists an open covering of $\R^\bullet[\Delta]$, given by
$$\{D(\R[X_i])\}_{i\in {\bf n}}:={\RSpec((M_\Delta)_{X_i})}_{i\in {\bf n}}.$$

\begin{Th}
	\label{simplicialcomplexcompute}
	Let $\Delta$ denote a simplicial complex on the vertex set $ \{1, \ldots, n\} $. Then the fundamental groupoid of the real punctured spectrum of the Stanley-Reisner ring $\R^\bullet[\Delta] $ is given as the following groupoid: The objects are	
	\[    \{  F \text{ facet of } \Delta ,\, \sigma_F :F \rightarrow \{+,-\}  \}  . \]
	The isomorphisms are generated by
	\[	\alpha_{i,[\sigma_F,\tau_G]}:  (F, \sigma) \rightarrow (G, \tau)    \]
	whenever there exists $i \in F \cap G$ with $\sigma(i) = \tau(i)$. These are subject to the relations generated by:
	\begin{itemize}
		\item[R1)] For a fixed $i\in \{1,\ldots, n\}$, $\alpha_{i,[\sigma_F,\rho_H]}=\alpha_{i,[\tau_G,\rho_H]}\circ\alpha_{i,[\sigma_F,\tau_G]}$. That is, the diagram
		$$\xymatrix@C=4em{(F,\sigma)\ar[r]^{\alpha_{i,[\sigma_F,\tau_G]}}\ar@/^2pc/[rr]^{\alpha_{i,[\sigma_F,\rho_H]}} & (G,\tau)\ar[r]^{\alpha_{i,[\tau_G,\rho_H]}} & (H,\rho) }$$
		commutes;
		\item[R2)] For every $\{ i,j \} \subseteq F \cap G$ with $\sigma(i) = \tau(i)$ and $\sigma(j) = \tau(j)$, we have
		\[  \beta_{i,[\sigma_F,\tau_G]} = \beta_{j,[\sigma_F,\tau_G]}: (F,\sigma) \rightarrow (G, \tau) \]
		Here, $\beta$ denotes a concatenation of some $\alpha$-s.
	\end{itemize}
\end{Th}

\begin{proof}
	For each subset $\emptyset\neq I \subseteq \{1, \ldots, n\}$, define the groupoid $P(I)$, whose objects are
	\[ \{(F,\sigma),\, F \textnormal{ a facet of }\Delta \textnormal{ with } I \subseteq F \} \, \]
	with isomorphisms
	\[\alpha_{I,[\sigma_F,\tau_G]}: (F,\sigma) \rightarrow  (G,\tau) \text{,  whenever } \sigma  \text{ and } \tau \text{ agree on } I, \  \]
	subject to the relations
	$\alpha_{I,[\sigma_F,\rho_H]}=\alpha_{I,[\tau_G,\rho_H]}\circ\alpha_{I,[\sigma_F,\tau_G]}$. This is empty if $I$ is a nonface. For $I \subseteq J$, there are natural injections $P(J) \rightarrow P(I)$ sending $(F, \sigma)$ to itself and $\alpha_{J,[\sigma_F,\tau_G]}$ to $\alpha_{I,[\sigma_F,\tau_G]}$.
	
	The groupoid $P(I)$ is equivalent to the discrete groupoid whose objects are $ \{ +,- \}^{I} $. To see this, first observe that the relation $\alpha_{I,[\sigma_F,\rho_H]}=\alpha_{I,[\tau_G,\rho_H]}\circ\alpha_{I,[\sigma_F,\tau_G]}$ implies that $P(I)$ is a groupoid, as there is at most one isomorphism between any two objects. To see that it has exactly $2^I$ isomorphism classes of objects, let $I \subseteq F,G$ be two extensions of $\rho:I \rightarrow \{ +, -\}$ to facets. There is a unique isomorphism relating them. The group of units of the binoid $M_\Delta$, localised at the elements of $I$, is $\Z^I$ (which is $0$ if $I$ is a nonface) and the tensorisation with $\Z/(2)$ gives $ \{+,- \}^{I} $. This is equivalent to the real fundamental groupoid $\Pi_1 (D( \R[\prod_{i \in I}  X_i] ))$ by Theorem \ref{reductiontogroupsnonsep}.
	
	Hence, we may compute $\Pi_1 ( \bigcup_{i = 1}^n D(\R[X_i]) )\simeq\Pi_1(\R^\bullet[\Delta])$ as the 2-colimit of $P$. We claim that the conditions of Theorem \ref{colim=2-colim} are fulfilled, thereby allowing us to compute it as the colimit of $P$. As mentioned above, the maps $P(J) \rightarrow P(I)$, for $I \subseteq J$, are injective on objects. Moreover, for each collection $I_1, \ldots , I_s$, we have that  \[ \operatorname{colim}_{j} P(I_j) =  P(I_1) \cup \ldots \cup P(I_s)   \rightarrow  P(I_1 \cap \ldots \cap I_s) \] is injective on objects.
	
	For the computation of the colimit of $P$ only the subsets $I$ with $1$ or $2$ elements are relevant. Its objects consist of the set of all $(F, \sigma)$, with $F$ a facet of $\Delta$. The isomorphisms are generated by the generators of the isomorphisms of $P(\{i\} )$, modulo the relations coming from $P(\{i,j\})$.
\end{proof}

We give some easy examples of how to apply this method. For these examples, there are many other ways to compute the fundamental groupoid.

\begin{Exm}
	We consider the full simplicial complex on $\{1, \ldots ,n\}$. There is only the facet $F=\{1, \ldots ,n\} $
	and there are $2^n$ $\pm$-tuples. For $n=1$, the objects are $+$ and $-$ and there is only one automorphism for each object. For $n=2$, the generating isomorphisms are
	$$\xymatrix@C=1em@R=3em{ \pt\limits_{(+,+)}\ar@/^/[rr] & & \pt\limits_{(+,-)}, \pt\limits_{(-,+)}\ar@/^/[rr] & & \pt\limits_{(-,-)},\pt\limits_{(+,+)}\ar@/^/[rr]& & \pt\limits_{(-,+)},\pt\limits_{(+,-)}\ar@/^/[rr]& & \pt\limits_{(-,-)} }$$
	and the only relation is the identity. As an example, the isomorphism $(+,+)\longrightarrow (+,-)$ would be $\alpha_{1,[\sigma_F,\tau_F]}$ in our notation, where $F=\{1,2\}$ with $\sigma(1)=+,\sigma(2)=+$ and $\tau(1)=+,\tau(2)=-$. This groupoid is connected with 4 objects and 4 isomorphisms. Hence, it is equivalent to the groupoid with one element, together with one generating automorphism without relations. So, the fundamental group is $\Z$.
	
	For $n \geq 3$, we show that every generating isomorphism $\alpha_{i,[\sigma_F,\tau_F]}: (F, \sigma) \rightarrow (F,\tau)$ has a factorisation into isomorphisms, which are the restrictions of two subset isomorphisms. The maps $\sigma$ and $\tau$ agree at least on $i$. If they agree also on another spot, the statement is clear. If they agree only on $i$, we pick two spots $j,k$ where they differ. Then there exists an isomorphism $\varphi$ which fixes $i$ and $j$ and another isomorphism $\psi$ which fixes $i$ and $k$ such that  $\alpha_{i,[\sigma_F,\tau_F]}$ can be written as
	\[ (F, \sigma) \stackrel{ \varphi}{\rightarrow} (F, \rho) \stackrel{\psi }{\rightarrow} (F,\tau)  \, . \]
	Hence, the fundamental group is trivial.
\end{Exm}

\begin{Exm}
	We consider the simplicial complex on $\{1,2,3\}$, given by the facets $\{1,2\}, \{1,3\}, \{2,3\}$. There are twelve objects which are tupels of the form $(\pm,\pm,0)$ (and permutations), the generating isomorphisms are given whenever two such tuples coincide at a nonzero spot. As such, the groupoid can be visualised by the diagram
	
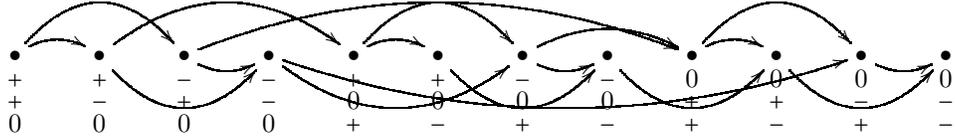
\vspace{2em}
	$$\xymatrix@C=0.7em@R=-0.4em{ \pt\ar@/^/[rr]\ar@/^2em/[rrrr] & & \pt\ar@/^2em/[rrrrrr]\ar@/_2em/[rrrr] & & \pt\ar@/^2em/[rrrrrrrrrrrr]\ar@/_/[rr] & & \pt\ar@/_2em/[rrrrrr]\ar@/_2em/[rrrrrrrrrrrrrr] & & \pt\ar@/^/[rr]\ar@/^2em/[rrrr] & & \pt\ar@/_2em/[rrrr] & & \pt\ar@/^1em/[rrrr]\ar@/_/[rr] & & \pt\ar@/_2em/[rrrr] & & \pt\ar@/^/[rr]\ar@/^2em/[rrrr] & & \pt\ar@/_2em/[rrrr] & & \pt\ar@/_/[rr] & & \pt\\ && && && && && && && && && && && \\		
	+ & & + & & - & & - & & + & & + & & - & & - & & 0 & & 0 & & 0 & & 0 \\
	+ & & - & & + & & - & & 0 & & 0 & & 0 & & 0 & & + & & + & & - & & - \\
	0 & & 0 & & 0 & & 0 & & + & & - & & + & & - & & + & & - & & + & & - }$$
	Here, the upper arrows are induced from the sign functions taking $+$ as their value, and the bottom arrows are induced by the $-$ sign.
	
	The relations coming from R1) are already taken into consideration in the above diagram, by not drawing in `unnecessary' relations in the first place. An example of this is the following: By the construction, we have isomorphisms connecting $(+,+,0)$, $(+,-,0)$ and $(+,0,+)$ in all pairwise combinations. However, one of them is simply the composition of the other two by R1), so we omit it in the diagram.
	
	The only relations coming from R2) are the identities. As such, we have 12 objects with 18 isomorphisms. This is equivalent to one object with $18-(12-1)=7$ automorphisms. The fundamental group of the real punctured spectrum is the free group with $7$ generators.
\end{Exm}

\begin{Rem}
	The intersection of the spectrum of a simplicial complex $\Delta$ with the hyperplane $\{(x_1 , \ldots, x_n),\, x_i \geq 0,\, \sum_{i=1}^n x_i=1\}$ gives a geometric realization of the simplicial complex. It is homotopy equivalent to the  positive octant of the spectrum (the intersection with $(\R_+)^n$). The homotopy type of the intersection with $D( \prod_{i \in I} X_i)$ is the homotopy type of $|I|$ points (only the $+$-sign occurs). Similar to Theorem \ref{simplicialcomplexcompute} we can compute the fundamental groupoid of the simplicial complex (its geometric realisation) with the following groupoid: the objects are the facets, the isomorphisms are generated by
	\[	\alpha_i: F \rightarrow G \]
	for every $i \in F \cap G$ and with relations
	\[  \alpha = \beta :F \rightarrow G \]
	whenever  there exist $\{ i,j \} \subseteq F \cap G$.
	This computation of the fundamental group of a simplicial complex is dual to the computation of it as the edge-path group of the simplicial complex.
\end{Rem}

\subsection{Example 1:}

Let $I=\{1,2,3\}$ and
\begin{eqnarray*} M_1&=&\gen{x_1,y_1,z_1}/x_1^2y_1=z_1^2,\\
	M_2&=&\gen{x_2,y_2,z_2}/x_2y_2=z_2^2,\\
	M_3&=&\gen{x_3,y_3}.
\end{eqnarray*}
The gluing isomorphisms are given as
\begin{eqnarray*} \lambda_{12}:(M_1)_{x_1}\rightarrow (M_2)_{x_2} &|&  x_1\mapsto x_2, \hspace{1.28em} y_1\mapsto x_2^{-1}y_2, \hspace{1em} z_1\mapsto z_2\\
	\lambda_{23}:(M_2)_{y_2}\rightarrow (M_3)_{x_3}  &|& x_2\mapsto x_3y_3^2, \ y_2\mapsto x_3^{-1}, \ \ \ \ \ \ z_2\mapsto y_3.
\end{eqnarray*}
We denote the binoid scheme obtained in this way by $X$. 

\subsubsection{The homology groups}

We have $\R X=U_1 \cup U_2 \cup U_3$, where $U_i=\RSpec(M_i)$. The homology of $\mathbb{R}X$ can be computed as the homology of the chain complex
$$\cdots\to 0\to H_0(U_{123})\xto{d_2} H_0(U_{12})\oplus H_0(U_{13}) \oplus H_0(U_{23}) \xto{d_1} H_0(U_1)\oplus H_0(U_2)\oplus H_0(U_3).$$
Since $U_{123}\simeq \RSpec(\Z\times \Z)\simeq U_{13}$, it follows that the map $H_0(U_{123})\to H_0(U_{13})$ is an isomorphism. Hence, the same homology can be computed by the chain complex
$$\cdots \to 0\to 0\to H_0(U_{12})\oplus  H_0(U_{23})
\xto{d} H_0(U_1)\oplus H_0(U_2)\oplus H_0(U_3).$$
Since $M_i^\times=\{1\}$, we have $H_0(U_i)=\Z$. We also have  $U_{12}\simeq \RSpec(\N\times \Z)\simeq U_{23}$. Thus,
$$H_0(U_{12})\simeq H_0(U_{23})=\Z^2$$
and the map $d:\Z^4\to \Z^3$ is given by the matrix 
$$\begin{pmatrix} 1&1&0&0\\-1&-1&1&1\\0&0&-1&-1\end{pmatrix}.$$
It follows that $$H_i( \mathbb{R}X)=\begin{cases}0, & {\rm if} \ i\geq 2,\\ \Z^2, & {\rm if} \ i =1,\\ 
\Z, & {\rm if} \ i=0.\end{cases}$$ 

\subsubsection{The fundamental groupoid}

We wish to calculate the fundamental groupoid. We will give a more detailed description, since this involves calculating the 2-colimit. While calculating the 2-colimit of an arbitrary 2-functor can be quite difficult, the situation can be significantly simplified when the source 2-category is a finite poset and the target only involves finite discrete groupoids.
\newline

\noindent
\underline{Step 1:} We first draw our binoid scheme $X$, as follows:
$$
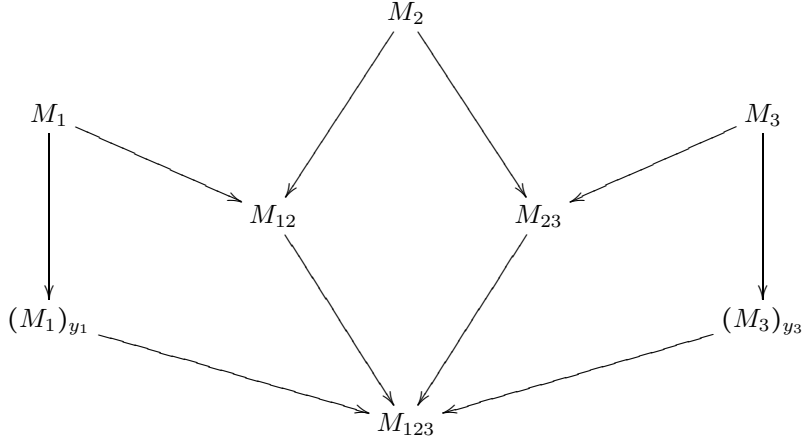
\xymatrix{ & & & M_2\ar[ddr]\ar[ddl] & & & \\
			 M_1\ar[dd]\ar[drr] & & & & & & M_3\ar[dd]\ar[dll] \\
			 & & M_{12}\ar[ddr] & & M_{23}\ar[ddl] \\
			 (M_1)_{y_1}\ar[drrr] & & & & & & (M_3)_{y_3}\ar[dlll] \\
			 & & & M_{123} & & & }$$
Here, $M_{12}:=(M_1)_{x_1}\simeq (M_2)_{x_2}$, $M_{23}:=(M_2)_{y_2}\simeq (M_3)_{x_3}$ and $M_{123}$ is the generic point. Open sets are of the form
$$U=\{M_{123}\}, U=\{M_2, M_{12}, M_{23}, M_{123}\} \textnormal{ or } U=\{(M_1)_{y_1}, M_{12}, M_{23}, M_{123}\}.$$
In other words $U\subseteq X$ is open, if $M_\alpha\in U$ implies that every localisation of $M$ is also in $U$.

A binoid scheme has a canonical covering. In this case, it is
$$U_1:=\{M_1, (M_1)_{y_1}, M_{12}, M_{123}\}, \quad U_2:=\{M_2, M_{12}, M_{23}, M_{123}\}$$
and
$$U_3:=\{M_3, M_{23}, (M_3)_{y_3}, M_{123}\}.$$
This gives us the following poset $\hat{C}_{\A}(X)$ of the \^Cech covering of $X$:
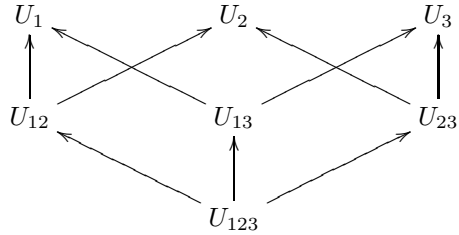
$$\xymatrix{ U_1 & & U_2 & & U_3 \\
	U_{12}\ar[u]\ar[urr] & & U_{13}\ar[ull]\ar[urr] & & U_{23}\ar[u]\ar[ull] \\
	& & U_{123}\ar[ull]\ar[urr]\ar[u] & & }$$
Clearly, $U_{13}=U_{123}=\{M_{123}\}$ as sets. However, we will treat $U_{13}$ and $U_{123}$ as distinct objects of the poset $\hat{C}_{\A}(X)$, in order to use Theorem \ref{colim=2-colim}.
\newline

\noindent
\underline{Step 2:} We calculate the fundamental groupoid of every $U_i$. We start by simplifying them using Thm. \ref{reductiontogroupsnonsep}, by replacing every $U_i\in\hat{C}_\A(X)$ with $\coprod_{r\in\Adm(M_i)} M_i(\r)^{\x}\otimes C_2^\circ$. Note that, as every $U_i$, $U_{ij}$ and $U_{ijk}$ has a maximal binoid, their global section is going to agree with said maximal binoid. That is, $\mathcal{O}_X(U_\alpha)=M_\alpha$. We require topological separatedness here, so that the intersections have a maximal binoid.

For this, we first determine the idempotents of $M_i$. It is easily seen that none of the binoids have any non-trivial idempotents. As such, we obtain:
$$\xymatrix{ 1^\circ\ar[d]\ar[drr] & & 1^\circ\ar[drr]\ar[dll] & & 1^\circ\ar[dll]\ar[d] \\
			 C_2^\circ(x_2)\ar[drr] & & C_2^\circ(x_2)\wedge C_2^\circ(z_2)\ar[d]^{\id}  & & C_2^\circ(y_2)\ar[dll]^\rho \\
	& & C_2^\circ(x_2)\wedge C_2^\circ(z_2). & & }$$
Here, $x_2$ in $C_2^\circ(x_2)$ just indicates the generator. The only non-obvious morphism $\rho:C_2^\circ(y_2)\rightarrow C_2^\circ(x_2)\wedge C_2^\circ(z_2)$ is given by $y_2\mapsto x_2^{-1}z_2^2=x_2$.
\newline

\noindent
\underline{Step 3:} We compose these groupoids with $\RSpec(-)=\Hom_{\mathsf{bn}}(-,\R)$. In the case of $C_2^\circ(x_2)$, for example, this is $\{x_2\mapsto 1, x_2\mapsto -1\}$. These are all clearly determined by the sign. Since these are individual points, the objects of $\Pi_1(\RSpec(C_2^\circ(x_2)))$ can be identified by $+x_2$ and $-x_2$. Similarly for all the other binoids. These groupoids are fully determined by this information as they have no non-trivial automorphisms. Hence, we can represent this as follows:
$$
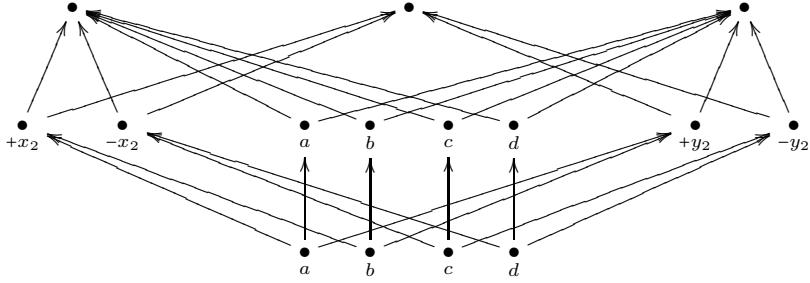
\xymatrix@C=.4em@R=3em{ & \pt & & & & & & & & & & \pt & & & & & & & & & & \pt &\\
						  \pt\limits_{+x_2}\ar[ur]\ar@[red][urrrrrrrrrrr] & & \pt\limits_{-x_2}\ar[ul]\ar@[red][urrrrrrrrr] & & & & & & \pt\limits_{a}\ar[ulllllll]\ar@[red][urrrrrrrrrrrrr] & & \pt\limits_{b}\ar[ulllllllll]\ar@[red][urrrrrrrrrrr] & & \pt\limits_{c}\ar[ulllllllllll]\ar@[red][urrrrrrrrr] & & \pt\limits_{d}\ar[ulllllllllllll]\ar@[red][urrrrrrr] & & & & & & \pt\limits_{+y_2}\ar[ulllllllll]\ar@[red][ur] & & \pt\limits_{-y_2}\ar[ulllllllllll]\ar@[red][ul]\\
						  & & & & & & & & \pt\limits_{a}\ar[ullllllll]\ar@[red][urrrrrrrrrrrr]\ar@[blue][u] & & \pt\limits_{b}\ar[ullllllllll]\ar@[red][urrrrrrrrrr]\ar@[blue][u] & & \pt\limits_{c}\ar[ullllllllll]\ar@[red][urrrrrrrrrr]\ar@[blue][u] & & \pt\limits_{d}\ar[ullllllllllll]\ar@[red][urrrrrrrr]\ar@[blue][u] & & & & & & & & }$$
	
Here, $\{\pt\limits_a \quad\pt\limits_b \quad\pt\limits_c \quad\pt\limits_d\}$ represents the groupoid $\Pi_1(C_2^\circ(x_2)\wedge C_2^\circ(z_2))$, with $a=(+x_2,+z_2)$, $b=(+x_2,-z_2)$, $c=(-x_2,+z_2)$ and $d=(-x_2,-z_2)$. The different colours are simply meant to help separate distinct functor. Since $\rho(y_2)=x_2$, we only need look at the sign of $x_2$ to determine the functor. For more on the induced functors between the groupoids see our discussion in Subsection \ref{morps}.
\newline

\noindent
\underline{Step 4:} The 2-colimit of the above diagram is equivalent to $\Pi_1(X)$, by Theorem \ref{Seifert-vanKampen}. We wish to simplify this calculation by using Theorem \ref{colim=2-colim}. For this, we need to replace the participating groupoids with equivalent ones, so that the conditions B1) and B2) will be satisfied.

In our case, $n=3$. For simplicity, we will write $\Pi_1(123)$ to mean $\Pi_1(\RSpec(M_{123}))$. More generally, we will use the notation $\Pi_1(12, 13, 123)$ to mean the colimit of the following diagram of groupoids:
$$\xymatrix{ \Pi_1(\RSpec(M_{12})) & &\\
			 \Pi_1(\RSpec(M_{123}))\ar[u]\ar[r] & \Pi_1(\RSpec M_{13}). }$$
Let us consider B1). We have $k=1$ or $k=2$. For $k=1$, we need to make the functor $\Pi_1(123)\rightarrow \Pi_1(23)$ injective on objects. For $k=2$, we need to make the functor $\Pi_1(13,23,123)\rightarrow \Pi_1(3)$ injective on objects.

Next, we consider B2). As $k=1$, we have $U\subseteq\{3\}$, as the condition $|U|\geq n-k-2$ trivialises. The case $U=\emptyset$ is the same as B1) for $k=1$. For $U=\{3\}$, we need to make the functor $\Pi_1(12)\rightarrow \Pi_1(2)$ injective on objects.

We need to be a bit careful with the ordering with which we make these 3 functors injective on objects. In particular, we need to concentrate on $\Pi_1(123)\rightarrow \Pi_1(23)$ before we move on to $\Pi_1(13,23,123)\rightarrow \Pi_1(3)$. However, it is easily seen that we can always choose a chronology with which we would not have to repeat our steps.

The induced diagram will look as follows:
$$
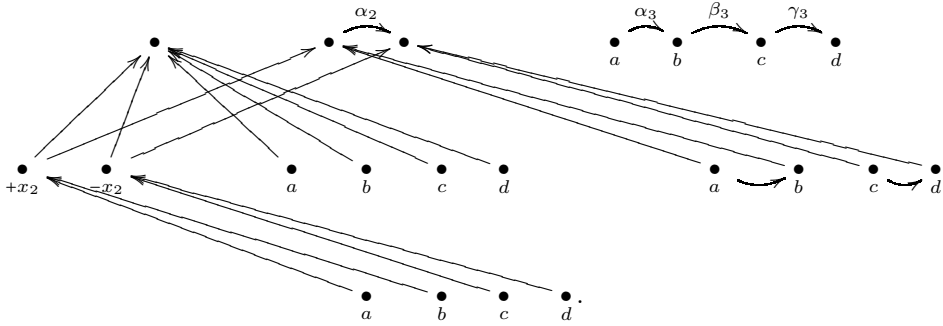
\xymatrix@C=.34em@R=3em{ & & & \pt & & & & & & \pt\ar@/^/[rr]^{\alpha_{2}} & & \pt & & & & & & \pt\limits_{a}\ar@/^/[rr]^{\alpha_{3}} & & \pt\limits_{b}\ar@/^/[rr]^{\beta_{3}} & & \pt\limits_{c}\ar@/^/[rr]^{\gamma_{3}} & & \pt\limits_{d} & & & \\
	\pt\limits_{+x_2}\ar[urrr]\ar@[red][urrrrrrrrr] & & \pt\limits_{-x_2}\ar[ur]\ar@[red][urrrrrrrrr] & & & & & & \pt\limits_{a}\ar[ulllll] & & \pt\limits_{b}\ar[ulllllll] & & \pt\limits_{c}\ar[ulllllllll] & & \pt\limits_{d}\ar[ulllllllllll] & & & & & & \pt\limits_{a}\ar[ulllllllllll]\ar@/_/[rr]\ & & \pt\limits_{b}\ar[ulllllllllllll] & & \pt\limits_{c}\ar[ulllllllllllll]\ar@/_/[rr] & & \pt\limits_{d}\ar[ulllllllllllllll]\\
	& & & & & & & & & & \pt\limits_{a}\ar[ullllllllll] & & \pt\limits_{b}\ar[ullllllllllll] & & \pt\limits_{c}\ar[ullllllllllll] & & \pt\limits_{d}.\ar[ullllllllllllll] & & & & & & & & }$$
Here, $\xymatrix@C=.5em{ \pt\ar@/^/[rr]^{-} & & \pt }$ denotes the connecting isomorphism between the two points. These morphisms must map to the connecting isomorphisms above them. This groupoid is clearly equivalent to the trivial groupoid $\pt$. Further, the endofunctors between $\xymatrix@C=.5em{\pt\limits_{a} & \pt\limits_{b} & \pt\limits_{c} & \pt\limits_{d} }$ have been omitted, as they are clear.
\newline

\noindent
\underline{Step 5:} The colimit and 2-colimit of the above diagram are equivalent by Proposition \ref{colim=2-colim}. The set of objects in the colimit of this diagram, is the colimit of the objects in this diagram.

Morphisms in the colimit are generated by the colimit of the morphisms, modulo the relations coming from the intersections.  We get the following as the 'generator' of our groupoid:
$$
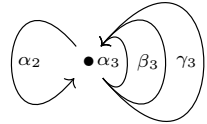
\begin{tikzcd}\pt
\ar[%
,loop % tells tikz-cd to do a loop
,out=130 % start at angle 123°
,in=230 % stop at angle 57°
,distance=5em % biggest distance of arrow to node. You can use pt or cm as well.
]{}{\alpha_2}
\ar[%
,loop % tells tikz-cd to do a loop
,out=310 % start at angle 123°
,in=50 % stop at angle 57°
,distance=2em % biggest distance of arrow to node. You can use pt or cm as well.
]{}{\alpha_3}
\ar[%
,loop % tells tikz-cd to do a loop
,out=310 % start at angle 123°
,in=50 % stop at angle 57°
,distance=5em % biggest distance of arrow to node. You can use pt or cm as well.
]{}{\beta_3}
\ar[%
,loop % tells tikz-cd to do a loop
,out=310 % start at angle 123°
,in=50 % stop at angle 57°
,distance=8em % biggest distance of arrow to node. You can use pt or cm as well.
]{}{\gamma_3}
\end{tikzcd}$$
The relations are the ones induced by the morphisms in the intersections. We have
\begin{eqnarray*} \alpha_2=\alpha_3=\gamma_3.
\end{eqnarray*}
As such, $\Pi_1(\R X)$ is equivalent to $\mathsf{F}_2$, the free group with 2 generators $\gen{\alpha_2,\beta_3}$, seen as a groupoid with one object.

\begin{center}
	
\end{center}
\end{document}